\documentclass[12pt,tablecaption=bottom]{amsart}
\usepackage[letterpaper, margin=1in]{geometry}

\usepackage{setspace}
\usepackage{amsmath}
\usepackage{bbold}
\usepackage{amssymb}
\usepackage{amsthm}
\usepackage{bm}
\usepackage{graphicx,color}
\usepackage{booktabs}
\usepackage[font=small,skip=0pt]{caption}
\usepackage{subcaption}
\usepackage{graphicx}
\usepackage{float}
\usepackage{amsmath}
\usepackage{enumitem}

\newtheorem{definition}{Definition}[section]
\newtheorem{lemma}{Lemma}[section]
\newtheorem{example}{Example}[section]
\newtheorem{theorem}{Theorem}[section]
\newtheorem{corollary}{Corollary}[section]
\newtheorem{proposition}{Proposition}[section]

\newtheorem{assumption}{Assumption}[section]

\title[Bayesian Chance Constrained Optimization]{Bayesian Chance Constrained Optimization: Approximations and Statistical Consistency}

\newcommand{\nX}{\mathbf X_n}
\newcommand{\cX}{\mathcal{X}}
\newcommand{\e}{\epsilon}
\newcommand{\vxi}{\bm{\xi}}
\newcommand{\x}{\mathbf x }

\newcommand{\In}{\mathbb I}

\newcommand{\bbE}{\mathbb{E}}

\newcommand{\cQ}{\mathcal{Q}}
\newcommand{\scKL}{\textsc{KL}}

\newcommand{\qnv}{q^*(\vxi|\nX)}

\newcommand{\rev}[1]{\textcolor{blue}{#1}}
\newcommand{\remove}[1]{\textcolor{red}{\sout{#1}}}
\renewcommand{\remove}[1]{\unskip}
\renewcommand{\rev}[1]{\textcolor{black}{#1}}
\newcommand{\revt}[1]{\textcolor{blue}{#1}}
\renewcommand{\revt}[1]{\textcolor{black}{#1}}
\newcommand{\removet}[1]{\textcolor{red}{\sout{#1}}}
\renewcommand{\removet}[1]{\unskip}

\author{Prateek Jaiswal}
\address{Department of Statistics, Texas A\& M University, 
College Station, TX 77483}
\email{jaiswalp@tamu.edu}

\author{Harsha Honnappa}
\address{School of Industrial Engineering, Purdue University, 
West Lafayette, IN 47906}
\email{honnappa@purdue.edu}

\author{Vinayak A. Rao}
\address{Department of Statistics, Purdue University, 
West Lafayette, IN 47906}
\email{varao@purdue.edu}

\date{}
\begin{document}

\maketitle

\begin{abstract}
  This paper considers data-driven chance-constrained stochastic optimization problems in a Bayesian framework. Bayesian posteriors afford a principled mechanism to incorporate data and prior knowledge into stochastic optimization problems. However, the computation of Bayesian posteriors is typically an intractable problem, and has spawned a large literature on approximate Bayesian computation. Here, in the context of chance-constrained optimization, we focus on the question of statistical consistency (in an appropriate sense) of the optimal value, 
computed using either an exact or approximate posterior distribution. To this end, we rigorously prove a frequentist consistency result demonstrating the convergence of the optimal value 
to that of a fixed, parameterized constrained optimization problem. We augment this by also establishing a probabilistic rate of convergence of the optimal value. 
We also prove the convex feasibility of the Bayesian stochastic optimization problem. Finally, we demonstrate the utility of our approach on an optimal staffing problem for an M/M/c queueing model.
\end{abstract}

\section{Introduction}
Consider a constrained optimization problem,
\vspace{-0.5em}
\begin{align*}\tag{TP}
\underset{x \in \mathcal X}{\min} &\quad f(\x,\vxi_0)\\
\text{s.t.}  &\quad g_i(\x,\vxi_0) \leq 0 , \ i \in \{1,2,3,\ldots, m\},
\end{align*}
where $\x \in \cX \subseteq \mathbb{R}^p$ is a decision vector in some convex set $\cX$ and $\vxi_0 \in \mathbb{R}^q$ parametrizes the problem.  The function $f: \cX \times  \mathbb R^q \mapsto \mathbb{R}$ encodes the cost/risk and the functions $g_i: \cX \times  \mathbb R^q \mapsto \mathbb{R}$ define the constraints. We assume that such a {\it nominal} optimization problem and its solution(s) exists, under suitable regularity conditions.

In practice, the parameter is often unknown beyond lying in some set $\Theta \subseteq \mathbb{R}^q$. 
It is natural, therefore, to assume the existence of a probability distribution $P(\cdot)$ with support $\Theta$ that quantifies the decision-maker's (DM) epistemic uncertainty about the parameter, leading to a {\it joint chance constrained} optimization problem
\vspace{-0.5em}
\begin{align*}\tag{JCCP}
\underset{x \in \mathcal X}{\min} &\quad E_P[f(\x,\vxi)]\\
\text{s.t.}  &\quad P\left(g_i(\x,\vxi) \leq 0 , \ i \in \{1,2,3,\ldots, m\} \right) \geq \beta.
\end{align*}
Note that a solution to (JCCP) is feasible for (TP) with probability at least $\beta$. Joint chance constrained problems 
have been used extensively to model a range of constrained optimization problems with parametric uncertainty~\cite{prekopa2003,nemirovski2006convex}. 


In this paper we are interested in data-driven settings where only a dataset $\nX$ of $n$ samples  -- so-called `covariates' -- is available, and whose joint distribution $P^n_{\vxi_0}(\cdot)$ depends on the `true' parameter $\vxi_0$. For instance, consider a staffing problem in a queueing system, where the goal is to compute the minimal number of servers required to ensure, with high probability, that the typical customer applying for service waits no more than a fixed amount of time to be served. The waiting time distribution for the typical customer depends on the arrival and service rates, which are unknown in a data-driven setting. Datasets here might include waiting times, inter-arrival and service times, whose distributions depend on the (unknown) rates. Problems of this type are prevalent across operations management~\cite{jiang2016data,Aktekin2016,Bandi2020}, finance~\cite{pagnoncelli2009computational}, and engineering~\cite{Liu15}.


In this data-driven setting, one might expect the epistemic uncertainty to diminish with an increasing number of samples, with each additional sample providing `new information' about the true parameter $\vxi_0$.
Bayesian methods provide a coherent way to quantify the devolution of the epistemic uncertainty 
through a {\it posterior} density $\pi(\vxi|\nX)$ over the parameters $\vxi \in \Theta$\footnote{For simplicity, we will assume the existence of density functions throughout this paper.}. 
The latter is computed by combining a {\it prior} density, quantifying {\it a priori} information (and biases) about the parameters, and a likelihood function, quantifying the probability of observed data under any parameter $\vxi$. 
Specifically, from Bayes' formula, it is well known that 
\begin{align}
    \pi(\vxi|\nX) = \frac{p^n_{\vxi}(\nX) \pi(\vxi)}{\int p^n_{\vxi}(\nX) \pi(\vxi) d\vxi},
    \label{eq:post}
\end{align}
where $\pi(\vxi)$ is the prior density, $p^n_{\vxi}(\nX)$ is the likelihood of observing $\nX$, and the denominator is the so-called data evidence. 
Bayesian methods have the advantage of calibrating uncertainty about hidden variables given partial observations.
Further, in many applications, incorporating prior knowledge is preferable to straight empirics. For example, in the queueing system design problem the {prior distribution} maybe specified by a modeler based on expert input and require that the arrival rate be strictly less than the total system capacity (ensuring that the system is stochastically stable). Of course, in the absence of such knowledge, uninformative priors (such as Jeffrey's prior or uniform priors) can be used and the same calculus holds.

This paper focuses on the formulation of a {\it Bayesian joint chance constrained program} (BJCCP) model, wherein a posterior distribution is used as the measure of epistemic uncertainty in (JCCP) to obtain,
\begin{align*}\tag{BJCCP}
\underset{x \in \mathcal X}{\min}  &\quad E_{\pi(\vxi|\nX)}[f(\x, \vxi)]\\
\text{s.t.}  &\quad \Pi\left(g_i(\x,\vxi) \leq 0 , \ i \in \{1,2,3,\ldots, m\} | \nX\right) \geq \beta,
\end{align*}
where, for any set $A\subseteq \Theta$, $\Pi(A|\nX)=\int_{A}\pi(\vxi|\nX) d\vxi$.
The (BJCCP) formulation provides a principled way to combine data with parametric models of the uncertainty in~(JCCP). \removet{To the best of our knowledge, this formulation has not been considered in the literature on data-driven chance constrained optimization before and is, we believe, a useful addition to the growing toolbox of methodology for solving such problems; see Section~\ref{sec:lit}.} \rev{\removet{However, n}\revt{N}otice that the chance constrained in our setting is over the epistemic uncertainty unlike standard chance constrained problem formulation (see~\cite{nemirovski2006convex}). A typical problem in standard chance constrained literature is a version of (TP) with $f(x,\vxi_0)=\bbE_{P_{\vxi_0}}[\bar f(\x,u)]$ and $g_i(\x,\vxi_0)= \bbE_{P_{\vxi_0}}[ \In_{\{\bar g_i(\x, u) \leq 0\}} ]$, where $u$ is a random variable with true distribution $P_{\vxi_0}$ and $\bar f(\cdot,\cdot)$ and $\bar g_i(\cdot,\cdot)$ are some known random functions. The main objective in this setting is to develop an algorithm that computes an approximation to the optimal $\x$ by using samples of $u\sim P_{\vxi_0}$. Also, note that here $P_{\vxi_0}$ is a measure of aleatory uncertainty.} 

\revt{We would like to note that problem-specific instances of the Bayesian chance-constraint formulation have been studied before across a broad range of applications covering science, engineering and operations management. An
instance solving a staffing problem in a queuing system appeared in~\cite{Aktekin2016}. However, there the authors used a conjugate prior and approximated the chance constraint using samples from the posterior distribution using a Monte Carlo approach. Bayesian formulations quantifying epistemic uncertainty in data-driven constrained optimization problems also appeared in~\cite{lai2011mean} and \cite{chitsazan2015bayesian}. \cite{chitsazan2015bayesian} uses Bayesian chance constrained approach to control the epistemic uncertainty in  measuring flow velocity to design a hydraulic barrier.~\cite{lai2011mean} proposes a Bayesian approach to solve the Markowitz portfolio optimization problem, where the posterior quantifies the uncertainty in the unknown mean and variance of asset returns. (BJCCP), as defined above, generalizes these problem-specific formulations. }

\revt{A commonality among ~\cite{Aktekin2016,chitsazan2015bayesian,lai2011mean} is that they all use simplifying conjugate priors and likelihoods to obtain tractable posterior distributions.}
\removet{The posterior can be computed in closed-form under conjugacy assumptions.} However, these simplifying assumptions are restrictive and untenable for many application settings. 
The computation of the posterior under more general conditions is intractable, since the evidence cannot be easily calculated. Consequently, there is a substantial body of work on {\it approximate Bayesian computation} focused on the question of efficiently and accurately approximating the posterior distribution. Broadly speaking, there are two classes of methods in approximate Bayesian computation: sampling methods and optimization-based methods. Markov chain Monte Carlo (MCMC) is the canonical sampling method, where the objective is to design 
a stationary Markov chain whose invariant distribution is precisely the posterior distribution. Initializing the Markov chain in an arbitrary initial state, after a `burn-in' period the state of the designed Markov chain is (roughly speaking) a sample from a distribution that closely approximates the invariant/posterior distribution (where closeness is typically measured in terms of the total variation distance). MCMC, however, is known to suffer from high variance, complex diagnostics, and has poor scaling properties with the problem dimension~\cite{blei2017variational}. Furthermore, as we will show below, sample-based methods in chance constrained settings can produce non-convex feasible sets, even when the `true' problem is convex feasible. Coupled with the high variance of the methods, it may not be appropriate to use MCMC (or other sampling methods) to solve data-driven chance constrained problems like (BJCCP).

Variational Bayesian (VB) methods~\cite{blei2017variational}, in contrast, use optimization to compute an approximation to the posterior distribution from a class of `simpler' distribution functions (that does not, necessarily, contain the posterior) called the {\it variational family}, by minimizing divergence from the posterior distribution. Importantly, the posterior distribution being intractable, VB methods optimize a surrogate objective that lower bounds the divergence measure, and the optimizer of the surrogate is precisely the posterior distribution when the class of distributions includes it. The Kullback-Leibler divergence is a standard choice in VB methods~\cite{blei2017variational}, though there is increasing interest in $\alpha$-R\'enyi divergence as well~\cite{turner11} which yield approximations that have better support coverage. Broadly speaking, VB methods trade variance for bias; specifically, there is no sampling variance, but since the variational family does not contain the `true' posterior, there is often an unavoidable bias that is introduced. From the perspective of solving data-driven chance constrained stochastic optimization problems, this trade-off may be appropriate, since the approximation (under very general conditions, as we show) is often necessarily convex feasible. Consequently, we focus on Kullback-Leibler divergence-based VB methods and consider the question of asymptotic consistency (in the large sample limit) of the variational approximation (VBJCCP) to (BJCCP).
    
Besides proposing (BJCCP) and (VBJCCP) (see Section~\ref{sec:CCP} below), our primary contributions are to
\begin{enumerate}
    \item Demonstrate the convex feasibility of the joint chance constraint (VBJCCP) when the posterior distribution belongs to a `nice' class of distributions.
    \item Establish the `frequentist' statistical consistency of the 
    value of both (BJCCP) and (VBJCCP) in the limit of a large data-set and a single chance constraint. 
    \item Quantify the consistency results for 
    the value of both (BJCCP) and (VBJCCP), by establishing a probabilistic rate of convergence for a single chance constraint.
\end{enumerate}

Frequentist consistency of Bayesian methods demonstrate that the Bayesian posterior concentrates on the `true' parameter $\vxi_0$ of the data generating distribution in the large sample limit. Typically this is demonstrated by showing that the posterior converges weakly to a Dirac delta distribution concentrated at $\vxi_0$ in probability or almost surely under the data-generating distribution~\cite{Gh1997}. Here, we consider the frequentist consistency of the value 
of (VBJCCP), and establish convergence in probability results demonstrating the consistency of VB approximations in Theorem~\ref{thm:1} and a probabilistic rate of convergence in Theorem~\ref{thm:finite}. Furthermore, as direct corollaries, we can easily recover consistency and rates of convergence for (BJCCP). 
We note that in preliminary works~\cite{jaiswal2020statistical,jaiswal2020variational} we claimed almost sure frequentist consistency of the optimal value of (VBJCCP) under the general conditions considered here. However, we subsequently realized that almost sure convergence is not possible under those conditions, and the consistency result in this paper \remove{shows} establishes convergence in probability. The rate of convergence results, of course, are entirely new.

\subsection{Relevant Literature}~\label{sec:lit}
\removet{To the best of our knowledge, Bayesian models of data-driven chance constrained optimization have not been considered before in the literature. }  \revt{The idea of accounting for epistemic uncertainty using a Bayesian and chance-constraint formulation is understudied. As noted above, problem-specific instances of the general Bayesian formulation (BJCCP) have appeared in~\cite{Aktekin2016,chitsazan2015bayesian,lai2011mean} , where authors use simplifying modeling assumptions to compute a closed form posterior distribution.} On the other hand, we note that there is precedence for Bayesian formulations of data-driven stochastic optimization problems -- for instance,~\cite{wu2018bayesian} develop the so-called Bayesian risk optimization (BRO) decision-making framework and establishes frequentist consistency of optimal values in the large sample limit; see recent follow-on work~\cite{cakmak2021solving,lin2021bayesian} as well. In~\cite{jaiswal2019b}, an approximate Bayesian formulation of the risk-sensitive decision-making problem is considered and, again, frequentist consistency results are established. None of these papers consider the chance constrained setting of this paper.

Nonetheless, there is an extensive literature on data-driven methods for solving chance constrained optimization problems, specifically scenario-based (SB)~\cite{calafiore2006distributionally,calafiore2006scenario,campi2009notes}, distributionally robust optimization (DRO)~\cite{calafiore2006distributionally,xie2019distributionally,jiang2016data,hota2019data,Gupta2019} and sample average approximation (SAA)~\cite{luedtke2008sample,pagnoncelli2009sample} approaches. This is by no means a comprehensive literature review, but highlights the range of approaches that have been explored. We direct the reader to the excellent recent review paper~\cite{geng2019data} for a comprehensive overview of the literature on data-driven chance constrained optimization. In particular, we observe that the ambiguity set in DRO quantifies the epistemic uncertainty when `centered' (defined, for instance, through the Wasserstein metric) around the empirical measure, which  converges to the data-generating measure in the large sample limit; see~\cite{cherukuri2020consistency} which establishes the consistency of chance-constrained DRO with Wasserstein ambiguity sets. This highlights an important difference with our current setting, where the posterior distribution (or its approximation) is used as a quantification of the epistemic uncertainty about the `true' parameter $\vxi_0$, and is shown to weakly converge to a Dirac delta distribution concentrated at $\vxi_0$, in the limit of a large sample of the covariates $\nX$. \rev{Another interesting paper~\cite{Gupta2019} considers a similar problem as (TP) with the unknown true model parameter and proposes a DRO framework, where ambiguity set of model parameters is constructed using the posterior distribution. The proposed DRO method in~\cite{Gupta2019} computes a robust optimal decision by maximizing the cost function over all possible models in the ambiguity set and then minimizing over all possible high-probability feasibility set, where high-probability feasibility set of decisions are computed using the ambiguity set and the constraint functions. }


The rest of the paper is laid out as follows. In the next section we introduce necessary notation and definitions that will be used throughout the paper. In Section~\ref{sec:CCP} we detail both (BJCCP) and (VBJCCP) providing a clean rationale for the modeling framework, and demonstrate the convex feasibility of (VBJCCP). \remove{Next, in Section~\ref{sec:consistent} we first establish the asymptotic consistency of the optimal value and the optimizers of (VBJCCP) under general conditions on the objective and constraint functions and then establish convergence rates for values of (VBJCCP) and (BJCCP).} \rev{Next, in Section~\ref{sec:consistent} we first establish the convergence rates for values of (VBJCCP) and (BJCCP) under general conditions on the objective and constraint functions and then asymptotic consistency of the optimal value and the optimizers of (VBJCCP)}. We end in Section~\ref{sec:App} with a simulation result demonstrating the efficacy of our approach in solving an optimal staffing problem.

\section{Notations and Definitions}

In this section, we introduce important notations and definitions used throughout the paper. 
We define an indicator function for any arbitrary set $A$ as $\In_{A}(t) : =  1 \text{ if $t\in A $ or } 0 \text{ if $t \notin A$ } $. Let $\|\cdot\|$ denote the Euclidean norm. Let $\delta_{\vxi}$ represent the Dirac delta distribution function, or singularity, concentrated at the parameter $\vxi$. Given an ensemble of random variables $\nX$ distributed as $P_0^n$ for any $n \geq 1$, following~\cite{GGV} we define the convergence of a sequence of random mappings $\{f_n:\nX\to \mathbb{R}\}$ to $f$ in $P_0^n$- probability as $\lim_{n\to\infty} P_0^n(|f_n-f|>\epsilon) = 0 $ for any $\epsilon>0$. We also use the notation $\lim_{n\to\infty} f_n \overset{P_0^n}{=} f$ or $f_n \overset{P_0^n}{\to} f$  as $n \to \infty$ to denote convergence in $P_0^n$- probability.   Next, we define degenerate distributions as
\sloppy
 \begin{definition}[Degenerate distributions]\label{def:degen}
 A sequence of distributions $\{q_n(\vxi)\}$ converges weakly to $\delta_{\vxi'}$ that is, $q_n(\vxi) \Rightarrow \delta_{\vxi'} $ for a $ \vxi' \in \Theta$,  if and only if  $\forall \eta > 0$
	\(	\lim_{n \to \infty } \int_{\{\|\vxi -\vxi'\| > \eta\}} {q}_n(\vxi) 	d\vxi = 0.\)
 \end{definition}

\begin{definition}[Rate of convergence]\label{def:roc}
    A sequence of distributions $\{q_n(\vxi) \}$ converges weakly to $\delta_{\vxi_1}$, $\forall \vxi_1 \in \Theta$ at the rate of $\gamma_n$ if 
    \begin{enumerate}
        \item[(1)] the  sequence of means $ \{\check \vxi_n := \int \vxi q_n(\vxi) d\vxi \}$ converges to $\vxi_1$ as $n\to \infty$, and 
        \item[(2)] the variance of $\{q_n(\vxi) \}$ satisfies
        \(E_{q_n(\vxi)}[\|\vxi - \check \vxi_n\|^2] = O\left (\removet{\frac{1}{\gamma_n^2}}\revt{\gamma_n^2}\right).\)
    \end{enumerate}
\end{definition}

We also define rescaled density functions as follows.
\vspace{.5em}
\begin{definition}[Rescaled density]
    For a random variable $\xi$ distributed as $d(\xi)$ with expectation $\tilde{\xi}$, for any sequence of matrices $\{t_n\}$, the density of the rescaled random variable $\mu := t_n (\xi - \tilde{\xi})$ is
    \(\check{d}_n(\mu) = |det(t_n^{-1})| d(t_n^{-1} \mu + \tilde{\xi}), \) where $det(\cdot)$ represents the determinant of the matrix.
\end{definition}\label{def:rescale}



Next, recall the definition of a test function \cite{Sc1965}.
\begin{definition}[Test function]\label{def:test}
    Let $\nX$ be a sequence of random variables on measurable space $(\mathbb{R}^{q\times n},\mathcal{S}^n)$. Then any $\mathcal S^n$-measurable sequence of functions $\{\phi_n\},~ \phi_n: \nX \mapsto [0,1]~\forall n \in \mathbb N$, is a \textit{test of a hypothesis} that a probability measure on $\mathcal S^n$ belongs  to a given set against the hypothesis that  it belongs to an alternative set. The test $\phi_n$ is \textit{consistent} for hypothesis $P_0^n$ against the alternative $P^n \in \{P_{\vxi}^n : \vxi \in \Theta\backslash\{\vxi_0\} \}$ if $\mathbb{E}_{P^n}[\phi_n] \to \In_{ \{\vxi\in \Theta\backslash\{\vxi_0\} \}  }(\vxi), \forall \vxi \in \Theta$ as $n \to \infty$, where $\In_{\{\cdot\}}$ is an indicator function. 
\end{definition}

A classic example of a test function is $\phi^{\text{KS}}_n = \In_{\{\text{KS}_n > K_{\nu}\}}(\theta)$ that is constructed using the Kolmogorov-Smirnov statistic $\text{KS}_n := \sup_t |\mathbb{F}_n(t) - \mathbb{F}_{\theta}(t)|$, where $\mathbb{F}_n(t)$ and $\mathbb{F}_{\theta}(t)$ are the empirical and true distribution respectively,  and $K_{\nu}$ is the confidence level. If the null hypothesis is true, the Glivenko-Cantelli theorem~\cite[Theorem 19.1]{vdV00} shows that the KS statistic converges to zero as the number of samples increases to infinity.


\section{Variational Bayesian Chance Constrained Optimization}~\label{sec:CCP}
Consider a parameterized joint probability distribution $P^n_{\vxi}$ over $\mathbb R^{d \times n}$, where $\vxi \in \mathbb R^q$ and let $p^n_{\vxi}(\cdot)$ represent the corresponding density. We observe a random sample $\nX := \{X_1,X_2,\ldots,X_n\}$ drawn from $P^n_{\vxi_0} \equiv P^n_0$. Note that $\nX$ need {\it not} be an independent and identically distributed (IID) sequence. Recall from~\eqref{eq:post} that the Bayesian approach computes a posterior over the unknown `true' parameter $\vxi_0$, giving rise to the Bayesian joint chance-constrained optimization problem (BJCCP). As noted in the introduction, there are the two significant challenges in solving (BJCCP):

\begin{enumerate}[leftmargin= *]
    \item[(i)] \textit{Computing the posterior distribution.} While in some cases conjugate priors can be used, this  is not appropriate in most problems. In general, posterior computation is intractable, and it is  
    the common motivation for using approximate Bayesian inference methods~\cite{blei2017variational} .
    
    \item[(ii)] \textit{Convexity of the feasible set.} 
    Observe that, even if the posterior distribution is computable, to qualify (BJCCP) as a convex program,  the feasible set,
    \begin{align}
    \{ \x \in \cX : \Pi \left(g_i(\x,\vxi) \leq 0 , \ i \in \{1,2,3,\ldots, m\}|\nX \right) \geq \beta \} 
    \label{eq:fs}
    \end{align} 
    must be convex. 
    However, it is possible that this set is not convex, even when the underlying constraint  functions $g_i(\x,\vxi), i \in \{1,2,\ldots m\}$  are (in $\x$) and, thus, finding a global optimum becomes challenging~\cite{Prkopa1995}. This raises the canonical question of when (VBJCCP) and (BJCCP) are convex feasible. 
\end{enumerate}
   
    Note that, if the constraint function has some structural regularity and the posterior distribution belongs to an appropriate class of distributions, then it can be shown that the feasible set in~\eqref{eq:fs} is convex. For instance, 
    \begin{proposition}~\cite[Theorem 2.5]{prekopa2003}
        If the constraint functions $g_i(\x,\mathbf y), i \in \{1,2,\ldots m\}$ for $\x \in \mathcal{X}$ and $ \mathbf y \in \mathbb{R}^q$ are quasi-convex in $(\x,\mathbf y)$ and $\vxi$ is a random variable with log-concave probability distribution, then the feasible set in (BJCCP) is convex. 
    \end{proposition}
    \begin{proof}
        The proof is a direct consequence of the result in Theorem 2.5 in~\cite{prekopa2003}.
    \end{proof}
    Furthermore,~\cite{Lagoa2005} showed that if the constraint function $g_i(\x,\vxi)$ is of the form $\{\mathbf  a^T \x\leq \mathbf b\}$, where $\vxi=(\mathbf a^T,\mathbf b)^T $ and has a symmetric log-concave density then with $\beta>\frac{1}{2}$ the feasible set in (BJCCP) is convex.
To address the posterior intractability, Monte Carlo (MC) methods offer one way to do approximate Bayesian inference with asymptotic guarantees. 
However, their asymptotic guarantees are offset by issues like poor mixing, large variance and complex diagnostics in practical settings with finite computational budgets~\cite{Kass1998,Andrieu2003}. Apart from these common issues, there is another important reason due to which any sampling-based method cannot be used  directly to solve (BJCCP): using the empirical approximation to the posterior distribution (constructed using  the samples generated from MCMC algorithm) to approximate the chance-constraint feasible set in (BJCCP), results in a non-convex feasible set~\cite{PenaOrdieres2019SolvingCP}. To illustrate this, consider the following simple example of a chance-constraint feasible set motivated by~\cite{PenaOrdieres2019SolvingCP}. 
\begin{example}~\label{ex:Gaussian}
Figure~\ref{fig:MCVB}(a) plots the chance-constraint feasible set
\begin{align}  \left\{\x \in R^2: \mathcal{N} \left( \vxi^T\x -1 \leq 0 | \bm {\mu}= [0,0]^T  ,\mathbf \Sigma_A= [1,-0.1;-0.1,1]  \right) > \beta \right \} ,
\label{eq:eq1}
\end{align}
and its empirical approximator using 8000 MCMC samples (Metropolis-Hastings with a `burn-in' of 3000 samples) generated from the underlying correlated multivariate Gaussian distribution. We fix $\beta=0.9$. We observe that the  resulting MC approximate feasible set is non-convex.  


\textit{Next,  we show that using the popular `mean-field variational family'~\cite{blei2017variational} to approximate the correlated multivariate Gaussian distribution in the same example in~\eqref{eq:eq1}, we obtain a smooth and convex approximation to the (BJCCP) feasible set. First, we compute mean-field approximation $q_A(\vxi)$ and $q_B(\vxi)$ of $\mathcal{N} \left( \vxi  | \bm {\mu} =[0,0]^T  ,\mathbf \Sigma  \right)$ for four different covariance matrices $\mathbf \Sigma$, with fixed variance $\sigma_{11}=\sigma_{22}=1$ but varying covariance $\sigma_{12}=\{-0.1,-0.025,0.025,0.1\}$.
Then, we plot the respective approximate VB chance-constraint feasibility region  
in Figure~\ref{fig:MCVB}. We observe that VB approximation provides a smooth convex approximation to the true  feasibility  set, but it could be outside the true feasibility region if the $ \xi_1$ and $\xi_2$ are positively correlated.}

\begin{figure}
     \centering
     \begin{subfigure}[b]{0.24\textwidth}
         \centering
         \includegraphics[width=\textwidth,height=\linewidth]{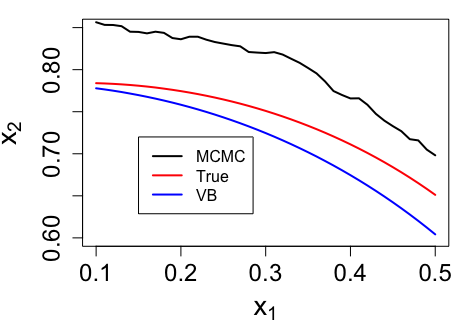}
         \caption{$\sigma^A_{12}=-0.1$}
         \label{fig:c1}
     \end{subfigure}
     \begin{subfigure}[b]{0.24\textwidth}
         \centering
         \includegraphics[width=\textwidth,height=\linewidth]{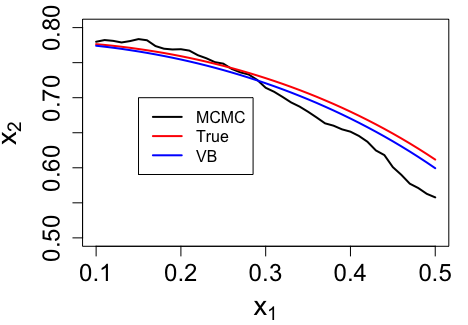}
         \caption{$\sigma^B_{12}=-0.025$}
         \label{fig:s1}
     \end{subfigure}
     \begin{subfigure}[b]{0.24\textwidth}
         \centering
         \includegraphics[width=\textwidth,height=\linewidth]{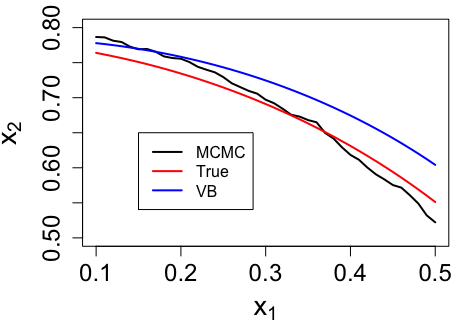}
         \caption{$\sigma^B_{12}=0.025$}
         \label{fig:c2}
     \end{subfigure}
     \begin{subfigure}[b]{0.24\textwidth}
         \centering
         \includegraphics[width=\textwidth,height=\linewidth]{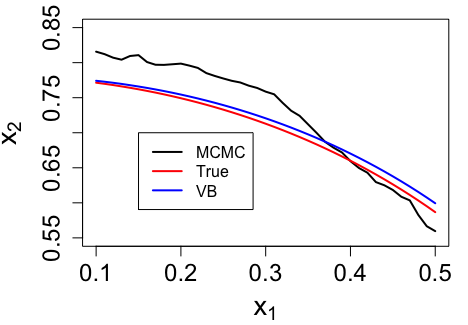}
         \caption{$\sigma^B_{12}=0.1$}
         \label{fig:s2}
     \end{subfigure}
        \caption{Feasible Region :  True Distribution vs Monte Carlo Approximation (5000 samples) vs. VB (mean field approximation).}
        \label{fig:MCVB}
\end{figure}
\end{example}

\subsection{Variational Bayes}
Variational Bayes (VB) methods are an alternative method for computing an approximate posterior. Standard VB minimizes the Kullback-Leibler (KL) divergence measure to compute $q^*$, the element in a given class of distributions 
$\cQ$ that is
`closest' to the posterior $\pi(\vxi|\nX)$:
\begin{eqnarray}
\qnv \in  \text{argmin}_{{q} \in \mathcal{Q}}~ 
\text{KL}({q}(\vxi)\|\pi(\vxi|\nX)) := \int q(\vxi) \log \frac{q(\vxi)}{\pi(\vxi|\nX)}d\vxi. \label{eq:vb_opt}
\end{eqnarray}
Using this, we approximate (BJCCP) with, 
\begin{align*}\tag{VBJCCP}
\underset{x \in \mathcal X}{\min} &\quad \bbE_{\qnv}[f(\x,\vxi)]\\
\text{s.t.}  &\quad Q^* \left(g_i(\x,\vxi) \leq 0 , \ i \in \{1,2,3,\ldots, m\}|\nX \right) \geq \beta,
\end{align*}
where $\beta$ is the confidence level and for any set $A\subseteq \Theta$, $Q^*(A|\nX)=\int_{A}\qnv d\vxi$. Observe that the optimization problem~\eqref{eq:vb_opt} is infeasible, since the posterior is unknown. However, unpacking the $\scKL$ divergence, we see that 
\begin{align}
\text{KL}({q}(\vxi)\|\pi(\vxi|\nX)) &= \int q(\vxi) \log \frac{q(\vxi)}{\pi(\vxi,\nX)}d\vxi + \log \int p^n_{\vxi}(\nX) \pi(\vxi) d\vxi.
\end{align}
Since, $\log \int p^n_{\vxi}(\nX) \pi(\vxi) d\vxi$ is a constant (with respect to $q$), minimizing the $\scKL$ divergence is equivalent to maximizing $\int q(\vxi) \log \frac{\pi(\vxi,\nX)}{q(\vxi)}d\vxi$. Since,~\scKL~divergence is non-negative, it follows that the log-evidence satisfies
\begin{align}
    \nonumber
    \log \int p^n_{\vxi}(\nX) \pi(\vxi) d\vxi &\geq \int q(\vxi) \log \frac{\pi(\vxi,\nX)}{q(\vxi)} d\vxi\\
    \tag{ELBO}
    &= -\text{KL}(q(\vxi)\| \pi(\vxi)) + \int \log  p^n_{\vxi}(\nX)~q(\vxi) d\vxi,
\end{align}
and the bound is tight if and only if the optimizer $q^*(\cdot)$ is the `true' posterior distribution. Thus, an approximate posterior can be computed by maximizing the so-called {\it evidence lower bound} (ELBO) in the final expression above:
\begin{align}
    q^*(\vxi | \nX) \in \underset{q \in \mathcal Q}{\arg\max} \int \log  p^n_{\vxi}(\nX)~ q(\vxi) d\vxi - \text{KL}(q(\vxi)\| \pi(\vxi)).
    \label{eq:ELBO}
\end{align}
 
Choosing the approximation to the posterior distribution from a class of `simple' distributions would facilitate in addressing the two critical problems associated with (BJCCP). Besides the  tractability of the posterior distribution, for instance, using the results in~\cite{prekopa2003} and~\cite{Lagoa2005} the choice of  a log-concave family of distributions as the approximating family could retain the convexity of the feasible set, if the constraint functions have certain structural regularity \rev{(see Proposition ~\ref{prop:2})}. \rev{However, we would also like to note that choosing a variational family is an important question in using any VB method. Often, in machine learning applications, the variational family is chosen based on computational convenience~\cite{blei2017variational}. Providing a general recipe to choose a variational family is challenging and an area of active research. }

As Example~\ref{ex:Gaussian} shows, the VB approximation of the feasibility set could include infeasible points, in general. This raises the question of whether the VB approximation can be consistent (in some appropriate sense) when the sample size $n$ is large. In other words, is there a notion of `frequentist' consistency of the feasibility set, the optimal values, and solutions? We address this question in the remainder of the paper.

\section{Asymptotic Analysis}~\label{sec:consistent}
In this section, we first identify regularity conditions on the prior distribution, the likelihood model, the variational family, and the risk and constraint functions to establish the rate at which the feasible region of (VBJCCP) coincides with the true feasible region. Then, under similar regularity conditions, we derive the convergence rate of the optimal value of (VBJCCP) to that of (TP), in the setting with a single constraint function (i.e., $m=1$). We derive the convergence rate result under very mild conditions on the  prior distribution and the likelihood models that are, nonetheless, hard to verify in practice for many problems of interest. Therefore, under more restrictive, but easily verifiable, regularity conditions we show that the 
the 
optimal \remove{values} \rev{value} $V^*_{VB}$ of  (VBJCCP)   
converges to the optimal value
$V^*$ of (TP) at $\vxi=\vxi_0$ (respectively), in $P_0^n-$probability as the number of samples converges to infinity, again in the setting with a single constraint function.

Note that it follows from the definition of the VB posterior $\qnv$ in~\eqref{eq:vb_opt} that when the variational family $\mathcal{Q}$ consists of all possible distributions then $\qnv$ coincides with the true posterior distribution. Consequently, all of our theoretical results for (VBJCCP) trivially extend to (BJCCP).

\subsection{Convergence rate and feasibility guarantee}
We state the assumptions under which we establish the rate of convergence and feasibility guarantee results. Let $L_n : \Theta \times  \Theta \mapsto [0,\infty)$ be an arbitrary loss function that measures the distance between parameters and also  depends on $n$. 
\begin{assumption}\label{assump:Asf1}
    Let $\{\epsilon_n\} \subset (0,\infty)$ be a sequence such that \remove{$\e_n \to \infty$ and } $n\e_n^2\geq 1$ \remove{ as $n\to\infty$} \rev{for all $n\geq 1$}. Fix $n\geq 1$. Then, for $L_n(\vxi,\vxi_0)\geq 0$ and any $\epsilon > \epsilon_n$, $\exists $ a test function $\phi_{n,\e} : \nX \mapsto [0,1]$ and sieve set $\Theta_n(\e) \subseteq \Theta$ such that 
    \begin{enumerate}
        \item[(i)] $\bbE_{P^{n}_0}[ \phi_{n,\e} ]  \leq C_0\exp(-C n \e^2 ), \text{ and }$
        \vspace{0.1em}
        \item[(ii)] $\underset{\{ \vxi \in \Theta_n(\e)  : L_n(\vxi,\vxi_0) \geq C_1 n \epsilon^2  \}}{\sup} \bbE_{P^{n}_{\vxi}}[ 1- \phi_{n,\e} ] \leq  \exp(-C n \e^2 )$.
    \end{enumerate}
\end{assumption}
\remove{Assumption \ref{assump:Asf1}$(i)$ quantifies the rate at which a Type-1 error diminishes with the sample size, while the condition in Assumption \ref{assump:Asf1}$(ii)$ quantifies that of a Type-2 error.} 
\rev{Since $\phi_{n,\e}$ is a test function (see Definition~\ref{def:test}) the LHS in Assumption~\ref{assump:Asf1}(i), the expectation of it under the null hypothesis ($\vxi=\vxi_0$), quantifies the probability of rejecting the null hypothesis when it is true, therefore, it is Type-I error. Similarly, the LHS in Assumption~\ref{assump:Asf1}(ii) is the expectation of the test function under the alternate hypothesis in the alternate set $\{ \vxi \in \Theta_n(\e)  : L_n(\vxi,\vxi_0) \geq C_1 n \epsilon^2  \}$, therefore it is Type-II error. Note that Assumption~\ref{assump:Asf1} is on the data generating model. Intuitively, it ensures the existence of a test $\phi_{n,\e}$ which is sufficiently powerful so that Type I/II errors decay at a certain rate. We will observe in Lemma~\ref{thrm:thm1} that the same rate governs the rate of convergence of the posterior distribution.} 
Assumption~\ref{assump:Asf2} below ensures the prior distribution places `sufficient' mass on the sieve set $\Theta_n(\e)$ defined in Assumption~\ref{assump:Asf1}.

\begin{assumption}\label{assump:Asf2}
    Let $\{\epsilon_n\} \subset (0,\infty)$ be a sequence such that \remove{$\e_n \to \infty$ and } $n\e_n^2\geq 1$ \remove{ as $n\to\infty$} \rev{for all $n\geq 1$}. Fix $n\geq 1$. Then, the prior distribution satisfies 
      $  \bbE_{\Pi}[ \In_{ \{ \Theta_n^c(\e) \} } ] \leq \exp(-C n \e^2). $
\end{assumption}

Notice that Assumption~\ref{assump:Asf2} is trivially satisfied if  $\Theta_n(\e) = \Theta$. The next assumption ensures that the prior distribution places sufficient mass around a neighborhood $A_n$, defined using the R\'enyi divergence, of the true parameter $\vxi_0$.
\sloppy
\begin{assumption}\label{assump:Asf3}
    Fix $n\geq 1$ and a constant $\lambda > 0$. Let   
    \(
    A_n := \left\{ \vxi \in \Theta :D_{1+\lambda} \left( P_0^n \| P_{\vxi}^n \right) \leq C_3 n \e_n^2 \right\}, 
    \)
    where $D_{1+\lambda} \left( P_0^n \| P_{\vxi}^n \right) := \frac{1}{\lambda} \log \int \left(\frac{dP_0^n}{dP_{\vxi}^n}\right)^\lambda dP_0^n $ is the R\'enyi divergence between $P_0^n$ and  $P_{\vxi}^n$, assuming $P_0^n$ is absolutely continuous with respect to $P_{\vxi}^n$. The prior distribution satisfies 
    \vspace{0em}
       $ \bbE_\Pi[\In_{\{A_n\}}]  \geq \exp(-n C_2 \e_n^2). $
\end{assumption}
Observe that the set $A_n$ defines a neighborhood of the distribution corresponding to $\vxi_0$.
If Assumption~\ref{assump:Asf3} is violated then the posterior too will place no mass in this neighborhood of $\vxi_0$, implying asymptotic inconsistency. 
Assumptions~\ref{assump:Asf1},~\ref{assump:Asf2}, and~\ref{assump:Asf3} are adopted from~\cite{GGV} and has also been used in~\cite{ZhGa2019} to prove convergence rates of variational posteriors. \rev{We also impose some standard regularity conditions on the cost and constraint functions.
\begin{assumption}\label{ass:Cath0}
    We assume that $f(\cdot,\vxi)$ and $g_i(\cdot,\vxi)$ are continuous for almost every $\vxi \in \Theta$.
    \end{assumption}
    }


Our main result demonstrating the rate of convergence follows a series of lemmas. All the proofs (except main results) can be found in~Section~\ref{sec:Proof4}.
We first recall the following result from \cite{ZhGa2019},
\begin{lemma}[Theorem 2.1~\cite{ZhGa2019}] \label{thrm:thm1}
    For any $L_n(\vxi,\vxi_0)\geq0$ and $\delta>0$, under Assumptions~\ref{assump:Asf1},~\ref{assump:Asf2},  and,~\ref{assump:Asf3}, and for $C > C_2 + C_3 +  2$ and $ \eta_n^2 := \frac{1}{n}\inf_{q \in \mathcal Q} \bbE_{P^{n}_0} \left[  \int_{\vxi}   	q(\vxi)  \log \frac{q(\vxi)} { \pi(\vxi | \nX)   } d\vxi  
    \right], $
    the VB approximator of the true posterior, $\qnv$, satisfies,
    \begin{align}
        P_0^n\left[ \int_{\vxi}  L_n(\vxi,\vxi_0) \qnv d\vxi > n \delta \right] \leq \frac{M }{\delta} ( \e_n^2 + \eta_n^2 )
        \label{eq:eq_lcb}
    \end{align}
    for some constant M that depends on the $C,C_1, C_2$, and $C_3$.
    
    
\end{lemma}

As noted before in Assumption~\ref{assump:Asf1}, the distance function $L_n(\vxi,\vxi_0)$ is arbitrary and it quantifies the distance between model $P_{\vxi}^n$ and $P_{0}^n$. For instance, $L_n(\vxi,\vxi_0)$ could be chosen to be $n \|\vxi-\vxi_0\|$. Also, note that the rate comprises of two sequences $\e_n^2$ and $\eta_n^2$.  The sequence $\e_n$ is the rate of convergence of the true posterior. In particular,~\cite{GGV} established $\e_n$ as the rate of convergence of the true posterior under Assumptions~\ref{assump:Asf1},~\ref{assump:Asf2},  and~\ref{assump:Asf3}. On the other hand, evident from its definition, the second sequence in the VB convergence rate is due to the variational approximation. Moreover, it is straightforward to observe that when $\mathcal{Q}$ is the family of all possible distributions, $\eta_n^2$ is $0$. Furthermore, under certain conditions on the variational family $\mathcal{Q}$ (see Assumption~\ref{assump:Asf11} ),  it can be shown that $\eta_n^2$ is bounded above by another convergent sequence $\e_n'^2$. In fact, in Lemma~\ref{lem:nv6} we show that $\e_n'=\e_n$ for the prior, the likelihood and the  variational family chosen for the optimal staffing problem discussed in Section~\ref{sec:App}. 

We first use the result above to prove the finite sample feasibility guarantee of the (VBJCCP) solution. Let us define the set where the true constraint $i\in\{1,2,\ldots m\}$ is satisfied as  \( F^i_0:=\{ \x \in \cX :  \{ g_i(\x,\vxi_0) \leq 0 \},   \}, \)
and VB-approximate feasible set is denoted as
\( \hat{F}_{VB}(\nX) := \{ \x \in \cX : Q^* \left(g_i(\x,\vxi) \leq 0 , \ i \in \{1,2,3,\ldots, m\}|\nX \right) \geq \beta \}. \) We show that the solutions obtained for (VBJCCP) are feasible for (TP) with  high probability. In particular, we show  that if  a point does not satisfy any of the constraints, then the probability of that point being in the VB approximate feasible set decays at a certain rate. We quantify  that  rate in  the following result.
\begin{theorem}~\label{prop:2}
    For any $i\in\{1,2,\ldots,m\}$ let $\x\in \cX \backslash F_0^i$ and \(L^i_n(\vxi,\vxi_0):= n \sup_{\x \in \cX} \In_{(0,\infty)}( g_i(\x,\vxi_0) - g_i(\x,\vxi))\)
    satisfies Assumption~\ref{assump:Asf1}. Then under Assumptions~\ref{assump:Asf2} and~\ref{assump:Asf3}, there exists a constant $C_i>0$ for each $i\in\{1,2,\ldots m\}$, such that
    \vspace{-1.2em}
    \begin{align*}
        P_0^n[ \x \in \hat{F}_{VB}(\nX)  ] \leq \frac{C_i}{\beta}(\e_n^2+\eta_n^2), 
    \end{align*}
    where $\e_n^2 \to 0$ as $n \to \infty$ and 
    \(\eta_n^2 = \frac{1}{n} \inf_{q \in \mathcal Q} \bbE_{P_0} \left[ \scKL(q(\vxi)\|\pi(\vxi | \nX) )\right].  \)
\end{theorem}

\begin{proof}
    Using Markov's inequality observe that for any  $\x\in \cX$,
    \begin{align}
        \nonumber
        P_0^n[ Q^* \left(g_i(\x,\vxi) \leq 0 , \ i \in \{1,\ldots, m\}|\nX \right) \geq \beta ] &\leq \frac{1}{\beta} \bbE_0[ Q^* \left(\cap_{i=1}^m \{ g_i(\x,\vxi) \leq 0\}|\nX \right) ]\\
        \leq   \frac{1}{\beta} &\bbE_0[ Q^* \left( \{ g_i(\x,\vxi) \leq 0\}|\nX \right) ] 
        \label{eq:eq2} 
    \end{align}
    for any $i \in \{1,\ldots,m\}$. Fixing $i \in \{1,\ldots,m\}$, since $\x\in \cX \backslash F^i_0$ implies that $\x \in \{ g_i(\x,\vxi_0) > 0 \}$, it follows that 
    \(\{ g_i(\x,\vxi) \leq 0\} \subseteq \{ g_i(\x,\vxi) < g_i(\x,\vxi_0)\}. \)
    Therefore, for all $\x\in \cX \backslash F^i_0$, it  follows from~\eqref{eq:eq2} that
    \begin{align}
        P_0^n[ Q^* \left(g_i(\x,\vxi) \leq 0 , \ i \in \{1,\ldots, m\}|\nX \right) \geq \beta ] 
        & \leq \frac{1}{\beta} \bbE_0[ Q^* \left(\{ g_i(\x,\vxi) < g_i(\x,\vxi_0)\}|\nX \right) ].
        \label{eq:p21}
    \end{align}
    Now using~\revt{Lemma~\ref{thrm:thm1}}, it follows that if \(L^i_n(\vxi,\vxi_0):= n \sup_{\x \in \cX} \In_{(0,\infty)}( g_i(\x,\vxi_0) - g_i(\x,\vxi))\)
    satisfies Assumption~\ref{assump:Asf1}, then there exists a constant $C_i$ such that 
    \(\bbE_0[ Q^* \left(\{ g_i(\x,\vxi) < g_i(\x,\vxi_0)\}|\nX \right) ] \leq C_i (\e_n^2+\eta_n^2),\)
    where $\eta_n^2:= \frac{1}{n} \inf_{q \in \mathcal Q} \bbE_{P_0} \left[  \int_{\vxi}   	q(\vxi)  \log \frac{q(\vxi)} { \pi(\vxi | \nX)   } d\vxi  
    \right]$. \removet{Finally, using Theorem~\ref{thrm:thm1} in~\eqref{eq:p21}, the assertion follows immediately}.
\end{proof}

Now, we state a straightforward corollary of the result above establishing feasibility guarantee of the (BJCCP) solution.
\begin{corollary}~\label{corr:1}
    For any $i\in\{1,2,\ldots,m\}$ let $\x\in \cX \backslash F_0^i$ and \(L^i_n(\vxi,\vxi_0):= n \sup_{\x \in \cX} \In_{(0,\infty)}( g_i(\x,\vxi_0) - g_i(\x,\vxi))\)
    satisfies Assumption~\ref{assump:Asf1}. Then under Assumptions~\ref{assump:Asf2} and~\ref{assump:Asf3}, there exists a constant $C_i>0$ for each $i\in\{1,2,\ldots m\}$, such that
    \[ P_0^n[ \x \in \hat{F}_{B}(\nX)  ] \leq \frac{C_i}{\beta}\e_n^2, \]
    where $ \hat{F}_{B}(\nX) := \{ \x \in \cX : \Pi \left(g_i(\x,\vxi) \leq 0 , \ i \in \{1,2,3,\ldots, m\}|\nX \right) \geq \beta \}$, $\e_n^2 \to 0$ as $n \to \infty$.
\end{corollary}
\begin{proof}
    The proof follows straightforwardly from Theorem~\ref{prop:2} and the fact that $\qnv$ is the same as the true posterior distribution and $\eta_n^2=0$, when  the variational family $\mathcal{Q}$ is fixed to the set of all possible distributions on $\Theta$.
\end{proof}

To leverage the result in~Lemma~\ref{thrm:thm1} in establishing the rate of convergence of the optimal value of (VBJCCP), we now fix $L_n(\vxi,\vxi_0)$ to specific positive distance functions in the following two lemmas. Lemma~\ref{lem:Gcons} establishes a rate of convergence of the VB posterior constraint set to the true constraint set. 
\begin{lemma}~\label{lem:Gcons}
    If $L^1_n(\vxi,\vxi_0) = n\sup_{\x\in \cX} |\In_{(-\infty,0]}(g(\x,\vxi))-  \In_{(-\infty,0]}(g(\x,\vxi_0)) |$ satisfies Assumption~\ref{assump:Asf1}, then under the conditions of Lemma~\ref{thrm:thm1}, for any $\delta>0$, we have
    \begin{align}
        P_0^n \left[ \sup_{\x \in \cX}|Q^* \left(g(\x,\vxi) \leq 0  |\nX \right) - \In_{(-\infty,0]}(g(\x,\vxi_0))| > \delta  \right]  \leq \frac{ M_1 }{\delta} ( \e_n^2 + \eta_n^2 ),
    \end{align}
    for a positive constant $M_1$.
\end{lemma}

In the following lemma, we establish the rate of convergence of the expected cost function  under VB posterior to  the true cost function. 
\begin{lemma}~\label{lem:Fcons}
    If $L^2_n(\vxi,\vxi_0) = n\sup_{\x\in \cX} |f(\x,\vxi)-  f(\x,\vxi_0) |$ satisfies Assumption~\ref{assump:Asf1}, then under conditions of Lemma~\ref{thrm:thm1} for any $\delta>0$,
    \begin{align}
        P_0^n[ \sup_{\x \in \cX} |\bbE_{\qnv}[f(\x,\vxi)] - f(\x, \vxi_0)| > \delta  ] \leq \frac{M_2 }{\delta} ( \e_n^2 + \eta_n^2 ).
    \end{align}
\end{lemma}

The next theorem proves a rate of convergence on the optimal value of (VBJCCP) as a consequence of the lemmas above.

\begin{theorem}\label{thm:finite}
    If $L^1_n(\vxi,\vxi_0) = n\sup_{\x\in \cX} |\In_{(-\infty,0]}(g(\x,\vxi))-  \In_{(-\infty,0]}(g(\x,\vxi_0)) |$ and  $L^2_n(\vxi,\vxi_0) = n\sup_{\x\in \cX}$ $ |f(\x,\vxi)-  f(\x,\vxi_0) |$ satisfy Assumption~\ref{assump:Asf1}, then under Assumption~\ref{assump:Asf2},~\ref{assump:Asf3},\rev{and~\ref{ass:Cath0}}, and when $\mathcal X$ is compact, for (fixed) constants $M_1>0$ and $M_2>0$ 
    , we have for any $\eta>0$ and $\delta\in (0,\beta)$
    \[ P_0^n[ |V^*_{VB}(\nX) -  V^*| > 2\eta  ] \leq \left[ \frac{M_1}{\min(\delta,1-\beta)}  + \frac{M_2}{\eta} \right]  ( \e_n^2 + \eta_n^2 ), \]
    where $\e_n^2 \to 0$ as $n \to \infty$ and
    \(\eta_n^2 : = \frac{1}{n} \inf_{q \in \mathcal Q} \bbE_{P_0} \left[  \int_{\vxi}   	q( \xi)  \log \frac{q(\xi)} { \pi(\xi | \nX)   } d\xi  
    \right].  \) 
\end{theorem}
\begin{proof}
    
    Recall $\mathcal S^*_{VB}(\nX) $ is the solution of (VBJCCP) and $\mathcal S^* $ is the solution of (TP) with $\vxi=\vxi_0$.
    Observe that, since both $Q^* \left(g(\x,\vxi) \leq 0  |\nX \right)$ and $\In_{(-\infty,0]}(g(\x,\vxi_0))$ are upper- semicontinuous \rev{due to Assumption~\ref{ass:Cath0} } their corresponding super-level sets are closed, and since $\cX$ is compact the corresponding feasible sets are also compact. Also, if the corresponding feasible sets  are non-empty then the corresponding optimal sets $\mathcal S^*_{VB}(\nX)$ and $\mathcal S^*$ are too.
    
    Next fix a point $\x^*$ in the true solution set of (TP).
    \remove{Since $\cX$ is compact, for any $\e>0$, there is $\x \in \cX$ such that for any $\e>0$, there exists $\x \in \cX$ such that $\|\x-\x^*\|<\e$ and $g(\x, \vxi_0)\leq 0$. This implies that there exists a sequence $\{\x_k\}\subset \cX$  such that $\x_k \to \x^*$ as $k \to \infty$ and $g(\x_k, \vxi_0)\leq 0$ for all $k\geq 1$.}
    Now \remove{fix $\x\in\cX $ such} \rev{noting} that $g(\x\rev{^*}, \vxi_0)\leq 0$ and, using Lemma~\ref{lem:Gcons}, \remove{observe} \rev{it follows} that \removet{for all  $n\geq n_0$}
    \begin{align*}
        P_0^n[ |Q^* \left(g(\x\rev{^*},\vxi) \leq 0  |\nX \right) &- \In_{(-\infty,0]}(g(\x\rev{^*},\vxi_0))| > \delta  ]\\ &=P_0^n[ |Q^* \left(g(\x\rev{^*},\vxi) \leq 0  |\nX \right) - 1| > \delta  ] 
        \leq \frac{M_1}{\delta} ( \e_n^2 + \eta_n^2 ).
    \end{align*}
    Now, fix $\beta\in(0,1)$ and let $\delta=1-\beta$. It follows from the  above inequality that, \removet{for all $n>n_0$}
      \[  P_0^n[ Q^* \left(g(\x\rev{^*},\vxi) \leq 0  |\nX \right) < 1-\delta  ] = P_0^n[ Q^* \left(g(\x\rev{^*},\vxi) \leq 0  |\nX \right) \leq \beta] \leq \frac{M_1}{1-\beta}( \e_n^2 + \eta_n^2 ).\]
    Notice that for $\x\rev{^*}\in\cX $ such that $g(\x\rev{^*}, \vxi_0)\leq 0$,
       \( \{ \remove{\x \in \mathcal X :} Q^* \left(g(\x\rev{^*},\vxi) \leq 0  |\nX \right) > \beta \} \subseteq \{ \remove{\x \in \mathcal X :} \bbE_{\qnv}[f(\x\rev{^*},\vxi)]\geq V_{VB}^*(\nX)   \}.\)
    Hence, \removet{for all $n\geq n_0$,}
        \begin{align}
        P_0^n[  \bbE_{\qnv}[f(\x\rev{^*},\vxi)]<V_{VB}^*(\nX) ]  \leq \frac{M_1}{1-\beta}( \e_n^2 + \eta_n^2 ).
        \label{eq:CR1}
    \end{align}
    Next, using the result in part(1) of~Lemma~\ref{lem:Fcons}, \removet{for all  $n\geq n_0$,} \remove{any} $\x=\rev{\x^*}\remove{\in\cX}$, and \rev{any} $\delta>0$
    \begin{align}
        P_0^n[ |\bbE_{\qnv}[f(\x\rev{^*},\vxi)] - f(\x\rev{^*}, \vxi_0)| > \delta  ] \leq \frac{M_2}{\delta}( \e_n^2 + \eta_n^2 ).
        \label{eq:CR2}
    \end{align}
    Observe that, for any $\eta>0$
    \begin{align*}
        &P_0^n[f(\x\rev{^*}, \vxi_0)- V_{VB}^*(\nX)  < -2\eta  ] 
        \\
        &\leq  P_0^n[  \bbE_{\qnv}[f(\x\rev{^*},\vxi)] - V_{VB}^*(\nX) < -\eta  ] 
        \\
        &\quad + P_0^n[ f(\x\rev{^*}, \vxi_0) - \bbE_{\qnv}[f(\x\rev{^*},\vxi)]< -\eta  ]
        \\
        &\leq  P_0^n[ \{  \bbE_{\qnv}[f(\x\rev{^*},\vxi)] -V_{VB}^*(\nX) <- \eta \}] + \frac{M_2}{\eta}( \e_n^2 + \eta_n^2 )
        \\
        &\leq  P_0^n[ \{  \bbE_{\qnv}[f(\x\rev{^*},\vxi)] - V_{VB}^*(\nX) < 0 \}] + \frac{M_2}{\eta}( \e_n^2 + \eta_n^2 )
        \\
        &\leq \frac{M_1}{1-\beta}( \e_n^2 + \eta_n^2 )  +\frac{M_2}{\eta}( \e_n^2 + \eta_n^2 )= \left[ \frac{M_1}{1-\beta} +\frac{M_2}{\eta}\right] ( \e_n^2 + \eta_n^2 ),
    \end{align*}
    where the second inequality follows from~\eqref{eq:CR2} and the last inequality uses~\eqref{eq:CR1}. \remove{Now, since $\x$ can be chosen arbitrarily close to  $\x^*$ } \rev{Therefore ,} it follows that
    \begin{align}
        P_0^n[ V^* - V_{VB}^*(\nX)  <- 2\eta  ] \remove{= P_0^n[ V^* - V_{VB}^*(\nX)< -2\eta  ]} \leq  \left[ \frac{M_1}{1-\beta} +\frac{M_2}{\eta}\right] ( \e_n^2 + \eta_n^2 ).
        \label{eq:eqF1}
    \end{align}
    
    Next, let $\hat\x_n \in \mathcal{S}_{VB}^*$; that is $\hat\x_n \in  \cX$, $Q^* \left(g(\hat\x_n,\vxi) \leq 0  |\nX \right) \geq \beta $ and $V_{VB}^*(\nX) = \bbE_{\qnv}[f(\hat\x_n,\vxi)]$. Since $\cX$ is compact, we assume that  $\hat\x_n \to \x_0$ (the limit point of the sequence $\{\hat\x_n\} \subseteq \cX$).
    
     Recall that Lemma~\ref{lem:Gcons} holds uniformly over any $\x \in \cX$, therefore using the fact that $Q^* \left(g(\hat\x_n,\vxi) \leq 0  |\mathbf X_n \right) -  \In_{(-\infty,0]}(g(\hat \x_n,\vxi_0)) \leq | Q^* \left(g(\hat\x_n,\vxi) \leq 0  |\mathbf X_n \right) -  \In_{(-\infty,0]}(g(\hat \x_n,\vxi_0))| \leq \sup_{\x\in \cX}|Q^* \left(g(\x,\vxi) \leq 0  |\mathbf X_n \right) -  \In_{(-\infty,0]}(g( \x,\vxi_0))|$, we have for \removet{all $n\geq n_0$ and} $\delta>0$, 
    \begin{align}
        P_0^n\left[  Q^* \left(g(\hat\x_n,\vxi) \leq 0  |\mathbf X_n \right) \leq  \In_{(-\infty,0]}(g(\hat \x_n,\vxi_0)) +\delta \right]\geq 1- \frac{M_1}{\delta} ( \e_n^2 + \eta_n^2 ).
    \end{align} 
    Next using the fact that $Q^* \left(g(\hat\x_n,\vxi) \leq 0  |\nX \right) \geq \beta $ for every $n\geq 1$, it follows 
    that $\hat \x_n$ is a feasible point of (TP) for $\delta\leq \beta$, that is
        $\left\{\remove{\x \in \mathcal{X} :} Q^* \left(g(\hat\x_n,\vxi) \leq 0  |\mathbf X_n \right) \leq  \In_{(-\infty,0]}(g(\hat\x_n,\vxi_0)) +\delta\right\} \subset \{\remove{\x \in \mathcal{X} : } \In_{(-\infty,0]}(g(\hat\x_n,\vxi_0)) +\delta \geq \beta \}$.
    Therefore, it follows that 
    \begin{align}
        \nonumber
        \bigg\{Q^* \left(g(\hat\x_n,\vxi) \leq 0  |\mathbf X_n \right) \leq  \In_{(-\infty,0]}(g(\hat\x_n,\vxi_0)) +\delta\bigg\} &\subseteq \{ \In_{(-\infty,0]}(g(\hat\x_n,\vxi_0)) +\delta \geq \beta \}
        \\
        &\subseteq \{ f(\hat\x_n,\vxi_0)\geq V^* \},
    \end{align}
    since the penultimate  condition implies that the $\hat\x_n$ is a feasible point of (TP). Therefore, for any $\delta\leq \beta$,
      \(  P_0^n\left[ f(\hat\x_n,\vxi_0)\leq V^* \right]\leq \frac{M_1}{\delta} ( \e_n^2 + \eta_n^2 ).\)
    Since Lemma~\ref{lem:Fcons} holds uniformly over all $\x$ and  therefore using the fact that $f(\hat\x_n,\vxi_0)-\bbE_{\qnv}[f(\hat\x_n,\vxi)] \leq | \bbE_{\qnv}[f(\hat\x_n,\vxi)] - f(\hat\x_n,\vxi_0)| \leq \sup_{\x\in \cX}|\bbE_{\qnv}[f(\x,\vxi)]-f(\x,\vxi_0)|$, for any $\delta>0$, we have 
       $ P_0^n\left[  \bbE_{\qnv}[f(\hat\x_n,\vxi)] +\delta \geq  f(\hat\x_n,\vxi_0)  \right] = P_0^n\bigg[  V_{VB}^*(\mathbf X_n) +\delta $ $\geq  f(\hat\x_n,\vxi_0)  \bigg] 
        \geq 1- \frac{M_2}{\delta} ( \e_n^2 + \eta_n^2 )$,
   and therefore
    \(
        P_0^n\left[  V_{VB}^*(\nX)+\delta \leq  f(\hat\x_n,\vxi_0)  \right] \leq  \frac{M_2}{\delta} ( \e_n^2 + \eta_n^2 ).\)
    Observe that for any $\eta>0$ 
    \begin{align}
        \nonumber
        P_0^n&\left[ V^*-  V_{VB}^*(\mathbf X_n) \geq  2\eta \right] 
        \\
        \nonumber
        &\leq P_0^n\left[ V^*-f(\hat\x_n,\vxi_0) \geq  \eta \right] 
        +P_0^n\left[  f(\hat\x_n,\vxi_0) -   V_{VB}^*(\mathbf X_n)\geq  \eta \right] 
        \\
        \nonumber
        &\leq P_0^n\left[ V^*-f(\hat\x_n,\vxi_0) \geq  0  \right] 
        +P_0^n\left[  f(\hat\x_n,\vxi_0) -   V_{VB}^*(\mathbf X_n)\geq  \eta \right] 
        \\
        &\leq \frac{M_1}{\delta} ( \e_n^2 + \eta_n^2 ) + \frac{M_2}{\eta} ( \e_n^2 + \eta_n^2 ) = \left[ \frac{M_1}{\delta}  + \frac{M_2}{\eta} \right] ( \e_n^2 + \eta_n^2 ),
        \label{eq:eqF2}
    \end{align} 
    where $\delta<\beta$.
    
    Combining equation~\eqref{eq:eqF1} and~\eqref{eq:eqF2}, we obtain
    \begin{align}
        \nonumber
        P_0^n\left[ |V^*-  V_{VB}^*(\mathbf X_n) | \geq  2\eta \right] & \leq \max \left( \left[ \frac{M_1}{\delta}  + \frac{M_2}{\eta} \right] , \left[ \frac{M_1}{1-\beta} +\frac{M_2}{\eta}\right]   \right) ( \e_n^2 + \eta_n^2 )
        \\
        & =  \left[ \frac{M_1}{\min(\delta,1-\beta)}  + \frac{M_2}{\eta} \right]  ( \e_n^2 + \eta_n^2 ).
    \end{align}
\end{proof}

The next result establishes the convergence rate of the optimal value of (BJCCP) with single constraint.
\begin{corollary}~\label{corr:finiteB}
    If $L^1_n(\vxi,\vxi_0) = n\sup_{\x\in \cX} |\In_{(-\infty,0]}(g(\x,\vxi))-  \In_{(-\infty,0]}(g(\x,\vxi_0)) |$ and  $L^2_n(\vxi,\vxi_0) = n\sup_{\x\in \cX}$ $ |f(\x,\vxi)-  f(\x,\vxi_0) |$ satisfy Assumption~\ref{assump:Asf1}, then under Assumption~\ref{assump:Asf2} and~\ref{assump:Asf3}, and when $\mathcal X$ is compact, for (fixed) constants $M_1>0$ and $M_2>0$ 
    , we have for any $\eta>0$ and $\delta\in (0,\beta)$,
    \( P_0^n[ |V^*_{B}(\nX) -  V^*| > 2\eta  ] \leq \left[ \frac{M_1}{\min(\delta,1-\beta)}  + \frac{M_2}{\eta} \right]   \e_n^2 , \)
    where $V^*_{B}(\nX)$ is the optimal value of (BJCCP) with single constraint and $\e_n^2 \to 0$ as $n \to \infty$.
\end{corollary}
\begin{proof}
    The proof is a direct consequence of Theorem~\ref{thm:finite} and the fact that $V^*_{VB}$ is the same as $V^*_{B}$ and $\eta_n^2=0$, when  the variational family $\mathcal{Q}$ is fixed to the set of all possible distributions on $\Theta$.
\end{proof}

\subsubsection{Characterizing $\eta_n^2$}~\label{sbsec:PropEta}
In order to characterize $\eta_n^2$,  we specify conditions on variational family $\cQ$ such that  
$\eta_n^2=O(\e_n'^2)$, for some $\e_n' \geq \frac{1}{\sqrt n}$ and $\e_n'\to 0$ as $n\to\infty$. We impose following condition on the variational family $\cQ$ that lets us obtain a bound on $\eta_n^2$.
\begin{assumption}\label{assump:Asf11}
    There exists a sequence of distributions $\{q_n(\cdot)\} \subset \cQ$ such that for a positive constant $C_1$,
    \(\frac{1}{n}  \left[ \scKL\left(q_n(\vxi)\|\pi(\vxi) \right) + \bbE_{q_n(\vxi)} \left[ \scKL\left(dP^n_{0}(\nX)\| dP^n_{\vxi}(\nX) \right) \right]  \right] \leq C_1 \e_n'^2. \)
\end{assumption} 
If the observations in $\nX$ are i.i.d, then observe that $ \frac{1}{n}    \bbE_{q_n(\vxi)} \left[ \scKL\left(dP^n_{0}(\nX))\| dP^n_{\vxi}(\nX) \right) \right]  = \bbE_{q_n(\vxi)} $ $\left[ \scKL\left(dP_{\lambda_0})\| dP_{\vxi}(\xi) \right) \right].  $
 Intuitively, this assumption implies that the variational family must contain a sequence of distributions that converges weakly to a Dirac delta distribution concentrated at the true parameter $\vxi_0$ otherwise the second term in the LHS of~Assumption~\ref{assump:Asf11} will be non-zero. We demonstrate the satisfaction of Assumption~\ref{assump:Asf11} for a specific variational family in Lemma~\ref{lem:nv6}.
%
\begin{proposition}\label{prop:eta_n}
    Under Assumption~\ref{assump:Asf11} and $C_9>0$,
    \vspace{0em}
    \(\eta_n^2 \leq  C_9\e_n'^2. \)
\end{proposition}

\subsubsection{Existence of Tests}

Recall that our convergence rates and finite sample feasibility guarantee depend on existence of certain tests for the specified distance functions. We prove a general result which is applicable to distance functions for which the set $\{\vxi \in \Theta: L_n(\vxi,\vxi_0)>n\e^2  \}$ is fixed for any $\e\in (0,1]$ and is a null set for any $\e>1$ (for example such distance functions should satisfy $n^{-1}L_n(\vxi,\vxi_0) \in \{0,1\}$). Notice that the distance functions $L_n^1(\vxi,\vxi_0)$ in Theorem~\ref{thm:finite} and  $L_n^i(\vxi,\vxi_0)$ in Theorem~\ref{prop:2}
satisfy these conditions.

We recall the following result from~\cite[Lemma 7.2]{GGV} which is due to Le Cam. 

\begin{lemma}~\label{lem:Lecam}
    Suppose that there exist tests $\omega_n$ such that for fixed sets $\mathcal{P}_0$ and $\mathcal{P}_1$, of probability measures
    \[ \sup_{P^n_0 \in \mathcal{P}_0 } \bbE_{P^n_0}[\omega_n] \to 0 \text{ and } \sup_{P^n \in \mathcal{P}_1 } \bbE_{P^n}[1-\omega_n] \to 0\text{ as } n \to \infty,\]
    then there exist tests $\phi_n$ and constants $K>0$ such that  
    \[ \sup_{P^n_0 \in \mathcal{P}_0 } \bbE_{P^n_0}[\phi_n] \leq e^{-Kn} \text{ and } \sup_{P^n \in \mathcal{P}_1 } \bbE_{P^n}[1-\phi_n] \leq e^{-Kn}.\]
\end{lemma}

\begin{proposition}~\label{prop:tests}
  Given $\Theta \subseteq \mathbb{R}^d$, if there exists a sequence of test function  $\phi'_{n,\e}$ for any $\e>1$, such that $\bbE_{P^n_0}[\phi'_{n,\e}] \leq e^{-K' n\e^2}$, then the distance functions $L_n^i(\vxi,\vxi_0)$ in Theorem~\ref{prop:2} for any $i\in \{1,\ldots,m\}$ and $L_n^1(\vxi,\vxi_0)$ in Theorem~\ref{thm:finite}  satisfy Assumption~\ref{assump:Asf1}.
\end{proposition}

\revt{A general recipe  for constructing $\phi'_{n,\epsilon}$ in Proposition~\ref{prop:tests} is as follows: construct an indicator function that restricts the domain of the true distribution to its tails, so that the required concentration bounds can be computed easily. We demonstrate this strategy to construct a sequence of test  functions with an example problem in Lemma~5.1.} Thus, for the distance function $L_n^2(\vxi,\vxi_0)$ in Theorem~\ref{thm:finite}, we have to either use~\cite[Lemma~7.1]{GGV} or construct an explicit test function to satisfy Assumption~\ref{assump:Asf1}. Interested readers may also refer to~\cite{GGV,ZhGa2019,jaiswal2019b} for further discussions on existence of tests and/or constructing bespoke test functions.

\subsection{Asymptotic consistency}
Although, the rate of convergence result implies asymptotic consistency, it will be evident from the application presented in~Section~\ref{sec:App} that the regularity conditions required to compute the rate are difficult to verify in practice. Consequently, in this section, we identify slightly more restrictive, but more easily verifiable, conditions on the  prior, likelihood, and  the variational family to guarantee asymptotic consistency of the optimal value and solution of (VBJCCP). We assume that $m=1$ in the remainder of this section.

First, we impose the following conditions on the prior distribution.
\begin{assumption}[Prior Density]~\label{assume:prior}
  \begin{enumerate}
    \item[(1)] The prior density function $\pi(\vxi)$ is continuous with non-zero measure in the neighborhood of the true parameter
    $\vxi_0$, and
    \item[(2)] there exists a constant $M_p > 0$ such that
    $\pi(\vxi) \leq M_p~\forall \vxi \in \Theta$ and $\bbE_{ \pi(\vxi)}[|\vxi|]< \infty$. 
  \end{enumerate}
\end{assumption}

Assumption~\ref{assume:prior} is satisfied by a large class of prior distributions. Next, we assume that the likelihood function satisfies the following asymptotic normality property. Recall that $P^n_0 \equiv P^n_{\vxi_0}$. 
\begin{assumption}[Local Asymptotic Normality]~\label{assume:lan}
  Fix $\vxi_0 \in \Theta$. The sequence of log-likelihood functions $\{ \log P^n_{\vxi}(\nX)\}$ satisfies a \emph{local asymptotic normality (LAN)} condition, if there exists a sequence of matrices $\{r_n\}$, a matrix $I(\vxi_0)$ and a sequence of random vectors $\{\Delta_{n,\vxi_0}\}$ weakly converging to $\mathcal{N}(0,I(\vxi_0)^{-1})$ as $n \to \infty$, such that for every compact set $K \subset \mathbb{R}^d$ 
  \[
  \sup_{h \in K} \left| \log \frac{{P^n_{\vxi_0 + r_n^{-1} h}}(\nX)}{{P^n_{\vxi_0}}(\nX)}  - h^T I(\vxi_0)
  \Delta_{n,\vxi_0} + \frac{1}{2} h^T I(\vxi_0)h \right| \xrightarrow{P^n_{0}} 0 \ \text{as $n  \to \infty$  }.
  \]
\end{assumption}
The LAN condition is standard, and holds for a wide variety of models. The assumption affords significant flexibility in the analysis by allowing the likelihood to be asymptotically approximated by a scaled Gaussian centered around $\vxi_0$~\cite{vdV00}. Any likelihood model that is twice-continuously differentiable satisfies the LAN condition~\cite[Eq. 7.15]{vdV00}.  Next, we place a restriction on the variational 
family $\cQ$:
\begin{assumption}~\label{assume:Var}
    \begin{enumerate}
   \item The variational family  $\cQ$ must contain distributions that are absolutely continuous with respect to the prior 
   distribution. 
   \item There exists a sequence of distributions $\{q_n(\vxi)\}$ in the variational family $\cQ$ that converges to a Dirac delta distribution $\delta_{\vxi_0}$ at the rate of $\sqrt{n}$ and with mean $\int \vxi q_n(\vxi) d\vxi = \hat \vxi_n$, the maximum likelihood estimate.
   \item The  differential entropy of the rescaled density (Definition~\ref{def:rescale}) of such sequence of distributions is positive and finite.
    \end{enumerate}
\end{assumption}
The first condition ensures that the $\scKL$ divergence in~\eqref{eq:vb_opt} is not undefined for all distributions in $\cQ$, that is not absolutely continuous with respect to the  posterior distribution. 
The Bernstein von-Mises theorem~\cite{vdV00} shows that under mild regularity conditions, the posterior converges to a Dirac delta distribution at the true parameter $\vxi_0$  at the rate  of $\sqrt{n}$, and the second condition ensures that the $\scKL$ divergence is well defined for all large enough $n$. 
These three assumptions together imply that the VB approximate posterior weakly converges to $\delta_{\vxi_0}$ as number of samples increases.   

\begin{lemma}[\cite{WaBl2017}]~\label{lem:VBcons}
     Under Assumptions~\ref{assume:prior},~\ref{assume:lan}, and~\ref{assume:Var}
    \(
        \qnv \in \underset{q \in \cQ}{\arg \min}~\scKL\left(
        q(\vxi)  \| \pi(\vxi | 
            \nX)
        \right) \Rightarrow \delta_{\vxi_0} \text{ in }~P_{0}^n-\text{probability as}~n
        \to \infty.\)
\end{lemma}
\begin{proof}
    See~\cite[Theorem 5(1)]{WaBl2017} for a proof.
\end{proof}
It must be noted that, as stated, the result in~\cite{WaBl2017} claims  $\qnv  \Rightarrow \delta_{\vxi_0}$ $P_{0}-$almost surely as $n\to \infty$. However, the proof of this result~\cite[Eqn 21-Supplementary material]{WaBl2017} can only establish convergence in $P_{0}^n-\text{probability as}~n \to \infty$. 
Now to establish asymptotic properties of the  optimal value and optimal solution to (VBJCCP), we assume that the following regularity conditions are satisfied by the cost and the constraint functions.
\begin{assumption}\label{ass:Cath}
    We assume that 
    \begin{enumerate}
        \item $f(\x,\cdot)$ and $g_i(\x,\cdot)$  are measurable and continuous for every $\x \in \cX$, and $f(\cdot,\vxi)$ and $g(\cdot,\vxi)$ are continuous for almost every $\vxi \in \Theta$. 
        \item $f(\cdot,\vxi)$ is locally Lipschitz continuous in $\x$ with for almost every $\vxi \in \Theta$, such that for $\x_1,\x_2$ in compact set $\cX$, $|f(\x_1,\vxi)-f(\x_2,\vxi)|\leq K_{\cX}(\vxi)\|x_1-x_2\|$ for some $K_{\cX}(\vxi)\leq \bar K_{\cX}$ for almost every $\vxi\in\Theta$.  
        \item $f(\x,\cdot)$ is uniformly integrable with respect to any $q$ in the  variational family $\cQ$, that is for any $\e>0$ and $\x\in \cX$, there exist a compact set $K_{\e}\subset \Theta$, such that $\int_{\Theta\backslash K_{\e}} f(\x,\vxi)q(\vxi)d\vxi<\e$.
    \end{enumerate}
\end{assumption}

We first establish consistency of the constraint function, under the `true' data generating distribution. 

\begin{lemma}\label{lem:pw}
    Under Assumptions~\ref{assume:prior},~\ref{assume:lan},and~\ref{assume:Var}, 
     we show that for any $\delta>0$ 
    \( P_0^n\left(\sup_{x\in \cX}\left|\bbE_{\qnv}\left[\prod_{i=1}^{m}\In_{(-\infty,0]}(g_i(\x,\vxi)\right] -  \prod_{i=1}^{m}\In_{(-\infty,0]}(g_i(\x,\vxi_0))\right|>\delta\right)\to 0\) as $n\to \infty$.
\end{lemma}


The next lemma establishes the point-wise and uniform convergence of the expected cost.
\begin{lemma}\label{lem:of}
    Under Assumptions~\ref{assume:prior},~\ref{assume:lan},~\ref{assume:Var}, and~\ref{ass:Cath}, we show that,
    \begin{enumerate}
        \item For each $\x \in \cX$,
        \(  \bbE_{\qnv}[f(\x,\vxi)]  \to f(\x, \vxi_0)~ \text{in} ~P_{0}^n-\text{probability as}~n
        \to \infty\).
        \item Suppose $\cX$ is compact, then $\sup_{x\in \cX}|\bbE_{\qnv}[f(\x,\vxi)]- f(\x,\vxi_0)|$ converges to 0 \text{in} $~P_{0}^n-\text{probability as}~n
        \to \infty$; that is for any $\delta>0$
        \(\lim_{n\to\infty}P_0^n\left( \sup_{x\in \cX} \left|\bbE_{\qnv}[f(\x,\vxi)]- f(\x,\vxi_0)\right|>\delta \right) =0. \)
    \end{enumerate}
\end{lemma}


Using the results in Lemma~\ref{lem:pw} and~\ref{lem:of},  Theorem~\ref{thm:1} establishes the asymptotic consistency  of the optimal values of (VBJCCP) and, as a consequence, (BJCCP) with single constraint. 
\begin{theorem}\label{thm:1}
    Under Assumptions~\ref{assume:prior},~\ref{assume:lan},~\ref{assume:Var}, and \ref{ass:Cath} and when $\cX$ is a compact set, we have $V^*_{VB}(\nX) \overset{P_0^n}{\to}  V^*$ as $n\to\infty$.
\end{theorem}

  \begin{proof}
    Recall $\mathcal S^*_{VB}(\nX) $ is the solution set of (VBJCCP) and $\mathcal S^* $ is the solution set of (TP). Observe that since both $Q^* \left(g(\x,\vxi) \leq 0  |\nX \right)$ and $\In_{(-\infty,0]}(g(\x,\vxi_0))$ are upper-semicontinuous, their corresponding super-level sets are closed and, since $\cX$ is compact, the corresponding feasible sets are compact. Furthermore, if the corresponding feasible sets are non-empty then the corresponding optimal sets $\mathcal S^*_{VB}(\nX)$ and $\mathcal S^*$ are also non-empty.
    
    Next fix a point $\x^*$ in the true solution set $\mathcal S^*$ of (TP). 
    \remove{Since $\cX$ is compact, for any $\e>0$, there is $\x \in \cX$ such that $\|\x-\x^*\|<\e$ and $g(\x, \vxi_0)\leq 0$. It follows that there exists a sequence $\{\x_k\}\subset \cX$  such that $\x_k \to \x^*$ as $k \to \infty$ and $g(\x_k, \vxi_0)\leq 0$ for all $k\geq 1$.
    Now fix $\x\in\cX $ such that} \rev{Note that} $g(\x\rev{^*}, \vxi_0)\leq 0$. By Lemma~\ref{lem:pw}, $Q^* \left(g(\x\rev{^*},\vxi) \leq 0  |\nX \right) \overset{P_0^n}{\to} $  $\In_{(-\infty,0]}(g(\x\rev{^*},\vxi_0))$ as $n\to \infty$, and therefore there exists an $n_0$ depending on $\e>0$  such that for all $n\geq n_0$ and any $\eta>0$, we have for a given confidence level $\beta\in(0,1)$,
    \begin{align*}
      P^n_0\Big( Q^* \left(g(\x\rev{^*},\vxi) \leq 0  |\nX \right)& \geq  \beta \Big) \geq P^n_0\left( Q^* \left(g(\x\rev{^*},\vxi) \leq 0  |\nX \right) \geq  1 \right) 
      \\
        \geq P^n_0&\left(   \In_{(-\infty,0]}(g(\x\rev{^*},\vxi_0)) - Q^* \left(g(\x\rev{^*},\vxi) \leq 0  |\nX \right) \leq   0 \right) 
      \\
       \geq P^n_0&\left(   \In_{(-\infty,0]}(g(\x\rev{^*},\vxi_0)) - Q^* \left(g(\x\rev{^*},\vxi) \leq 0  |\nX \right) \leq   -\eta \right) \geq 1-\e.
    \end{align*}
    Hence for all $n\geq n_0$, $\x\rev{^*}$ is a feasible solution  of (VBJCCP) with $P_0^n$-probability of at least $1-\e$, and therefore 
    \begin{align}
      P^n_0\left( \bbE_{\qnv}[f(\x\rev{^*},\vxi)]  \geq V_{VB}^*(\nX)\right) \geq P^n_0\left( Q^* \left(g(\x\rev{^*},\vxi) \leq 0  |\nX \right) \geq  \beta \right) \geq 1-\epsilon.
      \label{eq:p5}
    \end{align} 
    \remove{Now, since $\x$ can be chosen arbitrarily close to  $\x^*$, it follows from the  equation above and the bounded convergence theorem that}
    \begin{align*}
        \remove{P^n_0\left( \bbE_{\qnv}[f(\x^*,\vxi)]  \geq V_{VB}^*(\nX)\right) \geq 1-\epsilon }
    \end{align*}
    \remove{for all $n\geq n_0$.}
     For any $\delta>0$ observe that
    \begin{align*}
      P^n_0&\left( V^*_{VB}(\nX) - f(\x^*,\vxi_0) > \delta  \right)  
      \\&= P^n_0\left( V^*_{VB}(\nX)-\bbE_{\qnv}[f(\x^*,\vxi)] + \bbE_{\qnv}[f(\x^*,\vxi)]  -f(\x^*,\vxi_0) > \delta\right) 
      \\
      &\leq   P^n_0\left( V^*_{VB}(\nX)-  \bbE_{\qnv}[f(\x^*,\vxi)]   > \delta/2\right)
      \\
      &\quad +P^n_0\left( \bbE_{\qnv}[f(\x^*,\vxi)] -f(\x^*,\vxi_0) > \delta/2\right)
      \\
      &\leq   P^n_0\left(  V^*_{VB}(\nX)- \bbE_{\qnv}[f(\x^*,\vxi)]   > 0\right)
      \\
      & \quad+  P^n_0\left( \bbE_{\qnv}[f(\x^*,\vxi)]-f(\x^*,\vxi_0)  > \delta/2\right).
    \end{align*}
    By Lemma~\ref{lem:of}(1), for every $\x \in \cX$,
        \(  \bbE_{\qnv}[f(\x,\vxi)]  \overset{P_0^n}{\to} f(\x, \vxi_0)~ \text{as} ~n
        \to \infty\). Therefore it follows from the inequality above and ~\eqref{eq:p5} that
    \begin{align}
        \lim_{n\to \infty} P^n_0&\left( V^*_{VB}(\nX) - f(\x^*,\vxi_0) > \delta  \right)  =\lim_{n\to \infty} P^n_0\left( V^*_{VB}(\nX) - V^* > \delta  \right) = 0.
        \label{eq:eqF11}
    \end{align}
    
    We are left to show that $\lim_{n\to \infty} P^n_0\left( V^*_{VB}(\nX) - f(\x^*,\vxi_0) < - \delta  \right)  =\lim_{n\to \infty} P^n_0\left( V^*_{VB}(\nX) - V^* < -\delta  \right) = 0$ for any $\delta>0$. Let $\hat\x_n \in \mathcal{S}_{VB}^*$; that is $Q^* \left(g(\hat\x_n,\vxi) \leq 0  |\nX \right) \geq \beta $ and $V_{VB}^*(\nX) = \bbE_{\qnv}[f(\hat\x_n,\vxi)]$. Since $\cX$ is compact, we assume that as $n\to\infty$ $\hat\x_n \to \x_0 \in \cX$ (the limit point of the sequence $\{\hat\x_n\} \subseteq \cX$).
    
     Recall that Lemma~\ref{lem:pw} holds uniformly over all $\x \in \cX$. Therefore using the fact that $Q^* \left(g(\hat\x_n,\vxi) \leq 0  |\mathbf X_n \right) -  \In_{(-\infty,0]}(g(\hat \x_n,\vxi_0)) \leq | Q^* \left(g(\hat\x_n,\vxi) \leq 0  |\mathbf X_n \right) -  \In_{(-\infty,0]}(g(\hat \x_n,\vxi_0))| \leq \sup_{\x\in \cX}|Q^* \left(g(\x,\vxi) \leq 0  |\mathbf X_n \right) -  \In_{(-\infty,0]}(g( \x,\vxi_0))|$, we have for any $\eta>0$, 
      \[  \lim_{n\to\infty } P_0^n\left[  Q^* \left(g(\hat\x_n,\vxi) \leq 0  |\mathbf X_n \right) \leq  \In_{(-\infty,0]}(g(\hat \x_n,\vxi_0)) +\eta \right]= 1 .\]
    Next using the fact that $Q^* \left(g(\hat\x_n,\vxi) \leq 0  |\nX \right) \geq \beta $ for every $n\geq 1$, it follows 
    that $\hat \x_n$ is a feasible point of (TP) for $\eta\leq \beta$; that is,
        \(\left\{\remove{\x \in \mathcal{X} : }Q^* \left(g(\hat\x_n,\vxi) \leq 0  |\mathbf X_n \right) \leq  \In_{(-\infty,0]}(g(\hat\x_n,\vxi_0)) +\eta\right\} \subset \{\remove{\x \in \mathcal{X} :}  \In_{(-\infty,0]}(g(\hat\x_n,\vxi_0)) +\eta \geq \beta \}.\)
    Therefore, it follows that 
    \begin{align*}
       \left\{Q^* \left(g(\hat\x_n,\vxi) \leq 0  |\mathbf X_n \right) \leq  \In_{(-\infty,0]}(g(\hat\x_n,\vxi_0)) +\eta\right\} \subseteq \left\{ \In_{(-\infty,0]}(g(\hat\x_n,\vxi_0)) +\eta \geq \beta \right\}
        \\
        \subseteq \left\{ f(\hat\x_n,\vxi_0)\geq V^* \right\},
    \end{align*}
    since the penultimate  condition implies that $\hat\x_n$ is a feasible point of (TP). Therefore, for any $\eta\leq \beta$,
      \[  \lim_{n\to\infty} P_0^n\left[ f(\hat\x_n,\vxi_0)\leq V^* \right] = 0.\]
    Using the fact that $f(\hat\x_n,\vxi_0)-\bbE_{\qnv}[f(\hat\x_n,\vxi)] \leq | \bbE_{\qnv}[f(\hat\x_n,\vxi)] - f(\hat\x_n,\vxi_0)| \leq \sup_{\x\in \cX}|\bbE_{\qnv}[f(\x,\vxi)]-f(\x,\vxi_0)|$ and $\bbE_{\qnv}[f(\hat\x_n,\vxi)]=V_{VB}^*(\mathbf X_n)$, for any $\delta>0$ Lemma~\ref{lem:of}(2) implies that 
     \[   
     \lim_{n\to\infty} P_0^n\big[  V_{VB}^*(\mathbf X_n) +\delta \leq  f(\hat\x_n,\vxi_0)  \big] = 0.\]
    Observe that for any $\delta>0$ 
    \begin{align*}
        P_0^n &\left[ V^*-  V_{VB}^*(\mathbf X_n) \geq  \delta \right] 
        \\
        \nonumber
        &\leq P_0^n\left[ V^*-f(\hat\x_n,\vxi_0) \geq  \delta/2 \right] 
        +P_0^n\left[  f(\hat\x_n,\vxi_0) -   V_{VB}^*(\mathbf X_n)\geq  \delta/2 \right] 
        \\
        &\leq P_0^n\left[ V^*-f(\hat\x_n,\vxi_0) \geq  0  \right] 
        +P_0^n\left[  f(\hat\x_n,\vxi_0) -   V_{VB}^*(\mathbf X_n)\geq  \delta/2 \right] .
    \end{align*} 
    Taking limit $n\to\infty$ on either side of the inequality above, we  have \begin{align}
        \lim_{n\to\infty} P_0^n&\left[ V^*-  V_{VB}^*(\mathbf X_n) \geq  \delta \right]  = 0.
        \label{eq:eqF12}
    \end{align}
    Combining equation~\eqref{eq:eqF11} and~\eqref{eq:eqF12}, we conclude that for any $\delta>0$,
    \(    \lim_{n\to\infty} P_0^n\left[ |V^*-  V_{VB}^*(\mathbf X_n) | \geq  \delta \right] =0.\)
 \end{proof}

Next, we state the corollary of the result above that guarantees asymptotic consistency of the optimal value $V^*_{B}(\nX)$ of (BJCCP) with a single constraint.
\begin{corollary}\label{corr:thm1}
    Under Assumptions~\ref{assume:prior},~\ref{assume:lan},~\ref{assume:Var}, and \ref{ass:Cath} and when $\cX$ is a compact set, we have $V^*_{B}(\nX) \stackrel{P_0^n}{\to}  V^*~\text{as}~n\to\infty$.
\end{corollary}
\begin{proof}
The proof follows straightforwardly from Theorem~\ref{thm:1} and the fact that $V^*_{VB}(\nX)$ is the same as $V^*_{B}(\nX)$ when  the variational family $\mathcal{Q}$ is fixed to the set of all possible distributions on $\Theta$.
\end{proof}


\section{Application}~\label{sec:App}
Data-driven chance constrained optimization problems abound throughout operations research, finance, engineering and the sciences. In this section we present an example application of Bayesian chance constrained optimization to solving a staffing problem in a queueing system.

\subsection{Optimal Staffing} 
Consider a situation where a decision maker (DM) has to decide the optimal number of servers in a multi-server $M/M/c$ queueing system, using arrival time and service time data. We assume that the rate parameters of the exponentially distributed inter-arrival and service time distributions, denoted as $\lambda$ and $\mu$  respectively, are unknown. Note that $\lambda$ and $\mu$, together constitute the system parameter $\xi = \{\lambda ,\mu\}$  and  the number of servers $c$ is the decision/input variable.
 The DM collects $n$ realizations of the random vector $\mathcal{V}:=\{T,S,E\}$, denoted  as $\nX :=\{\mathcal{V}_1,\ldots \mathcal{V}_n\}$ where $T$, $S$, and $E$ are the random variables denoting the arrival, service-start, and service-end time of each customer $i \in \{1,2,\ldots  n\}$ respectively. We also assume that 
 the inter-arrival and service times are independent, that is $T_i-T_{i-1}$ is independent of $E_i-S_{i}$ for each  $i \geq 1$. The joint likelihood of the arrival and departure times for $n$ customers is \(p^n_{\vxi}(\nX) := \prod_{i=1}^{n} \lambda e^{-\lambda(T_i-T_{i-1})} \mu e^{-\mu(E_i-S_{i})} .\)

\textbf{Constraint functions:} The DM chooses the number of servers $c$ to maintain a constant measure of congestion. Congestion is usually measured as $1-W_q(c,\lambda,\mu)$, where $W_q(c,\lambda,\mu)$ is the steady-state probability that the customer did not wait in the queue. A closed-form expression for $1-W_q(c,\lambda,\mu)$ for an $M/M/c$ queue is known to be 
\( 1-W_q(c,\lambda,\mu) = \frac{r^c}{c!(1-\rho)} \Big/ \left(\frac{r^c}{c!(1-\rho)} + \sum_{t=0}^{c-1}\frac{r^t}{t!} \right), \)
where $r= \frac{\lambda}{\mu} \text{ and } \rho =\frac{r}{c}$ with $\rho<1$ (see~\cite{Gross2008}) . $\rho $ is also known as \textit{traffic intensity} and $\rho < 1$ is a necessary and sufficient condition for an $M/M/c$ queue to be in steady-state (or stable). 

The DM fixes $\alpha$, the desired maximum fraction of customers delayed in the queue and the smallest $c$ is chosen that satisfies
\(  ( \alpha - \{1-W_q(c,\lambda,\mu)\} ) > 0 \text{ and } (c \mu - \lambda) > 0.   \)
Referring to the queueing literature,  we will use the term the quality of service\textsc{(QoS)} constraint for the first constraint. In fact, the QoS constraint is only valid when $\rho<1$.  The corresponding constraint optimization problem is
\begin{align*}\tag{TP-Q}
\text{minimize} ~ c, 
\text{ subject to}~&~(\alpha - \{1-W_q(c,\xi)\} ) > 0 \textsc{(QoS)} \text{ and }
(c \mu - \lambda) > 0.
\end{align*}
This so-called staffing problem and its variants are well studied in the queueing literature. As noted before, we are interested in the data-driven setting where the parameters of the problem are unknown. This data-driven staffing problem has been considered as well and the interested reader may referred to~\cite{Gans2003} and~\cite{Aksin2009}. 

Next, we fix a non-conjugate inverse Gamma ($\text{Inv}-\Gamma(\cdot)$) distribution prior on both $\lambda$ and $\mu$, that is $d\Pi(\lambda,\mu)=\text{Inv}-\Gamma_{\lambda}(\lambda;\alpha_q,\beta_q)\text{Inv}-\Gamma_{\mu}(\mu;\alpha_s,\beta_s) d \lambda d \mu$. In our experiments, we fix $\alpha_q=\alpha_s=1$ and  $\beta_q=\beta_s=1$.  
We fix the variational family $\mathcal{Q}= \big\{ q(\lambda,\mu) :  q(\lambda,\mu;a_q, b_q,a_s, b_s) 
    = \Gamma(\lambda;a_q,b_q) \Gamma(\mu;a_s,b_s) \big\}$, where $\Gamma(\cdot;a_{(\cdot)},b_{(\cdot)})$ denotes the Gamma distribution with rate $b_{(\cdot)}$ and shape $a_{(\cdot)}$. In the simulation experiment, we fix $\lambda_0=16$ and $\mu_0=1$ and generate $2000$ samples of service and inter-arrival times. We then  solve the (VBJCCP) \remove{for 250 sample paths} and denote its solution as $C^*_{VB}$. We also solve the corresponding (BJCCP) using a sample average approximation (SAA) of the chance constrained problem, by generating samples from the posterior distribution using MCMC. We denote the optimal staffing level computed using  MCMC as $C^*_{MCMC}$. \rev{We repeat the experiment over 250 sample paths of service and inter-arrival times and compute respective $C^*_{VB}$ and $C^*_{MCMC}$. }
    
    The results of this simulation experiment are summarized in Figure~\ref{fig:VBvsMCMC}.
\begin{figure}[ht]
    \centering
         \begin{subfigure}[b]{0.45\textwidth} \includegraphics[trim={0 0.2cm 0 0},clip,width=0.85\textwidth,height=0.8\textwidth]{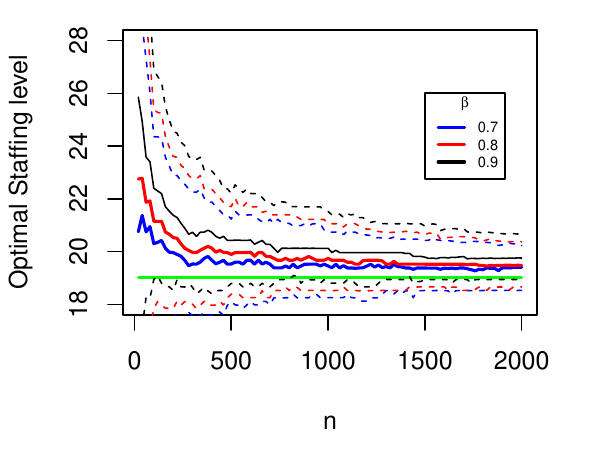}
        \caption{$C_{VB}^*$}
        \end{subfigure}
        \begin{subfigure}[b]{0.45\textwidth}
        \centering
         \includegraphics[trim={0 0.2cm 0 0},clip,width=0.85\textwidth,height=0.8\textwidth]{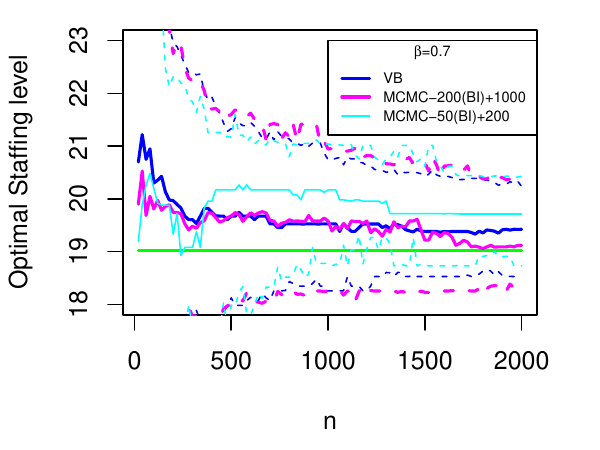}
        \caption{$C_{VB}^*$ vs. $C_{MCMC}^*$}
        \end{subfigure}
        \caption{$\lambda_0=16, \mu_0=1$, (a) Optimal Staffing Level ($5^{th}$, $50^{th}$, and $95^{th}$ quantile over 250 sample paths) for $\beta=\{0.7,0.8,0.9\}$ 
                    (b) $C_{VB}^*$ vs. $C_{MCMC}^*$ -Optimal Staffing Level ($5^{th}$, $50^{th}$, and $95^{th}$ quantile over 250 sample paths) against the number of samples ($n$)
                    , green line is the solution of (TP-Q) at $\{\lambda_0\,\mu_0\}$.}
                    ~\label{fig:VBvsMCMC}
        \end{figure}
We observe in Figure~\ref{fig:VBvsMCMC}(a) that $C^*_{VB}$ is consistent and moreover, for larger confidence level $\beta$, $C^*_{VB}$ is more conservative (i.e., the optimal number of servers is larger) as expected. In Figure~\ref{fig:VBvsMCMC}(b), we compare $C^*_{VB} $ and $C^*_{MCMC}$ for $\beta=0.7$. We compute $C^*_{MCMC}$ at each $n$ using two sequences of MCMC samples from the `true' posterior distribution generated using Metropolis–Hastings algorithm~\cite{Chib1995}: 1) 1000 samples with 200 burn-in (magenta) and 2) 200 samples with  50 burn-in (cyan).  
Observe that, as $n$ increases both $C^*_{VB}$ and $C^*_{MCMC}$ (magenta) converges to the true solution almost at the same rate and there is no significant difference between the two approaches. In fact, we will later show in~Theorem~\ref{thm:OSfinite} and~Corollary~\ref{corr:OSB} that the optimal staffing levels computed using the (VBJCCP) and (BJCCP) approaches converge at the same rate. Moreover, the average computation time taken by the VB and MCMC (magenta) approaches to compute an optimal staffing level at a given $n$ are of the same order (30 seconds (average) on Sky Lake CPU @ 2.60GHz). Unsurprisingly, the computation time in an MCMC approach can be reduced by reducing the number of samples;  however, it may result in computing a suboptimal solution. We observe that computing $C^*_{MCMC}$~(cyan) is faster (8 seconds (on average) on Sky Lake CPU @ 2.60GHz) but suboptimal. 

 Next, we verify the conditions on the prior, the likelihood model and the variational family to compute the convergence rate of $C^*_{VB}$. First note that the risk function  $f(c,\vxi)=c$ in the optimal staffing problem, therefore $L^2_n(\vxi,\vxi_0)$  is $0$. Hence, Lemma~\ref{lem:VBcons} is trivially true even without existence of tests conditions (Assumption~\ref{assump:Asf1}) defined using $L_n(\vxi,\vxi_0)= L^2_n(\vxi,\vxi_0)$. Next, we consider $L_n^1(\vxi,\vxi_0)$ and $L_n^i(\vxi,\vxi_0)$ for $i\in\{1,\ldots,m\}$ and recall Proposition~\ref{prop:tests}. We satisfy the conditions of Proposition~\ref{prop:tests} in  the  following result so that these distance functions satisfy Assumption~\ref{assump:Asf1}.

\begin{lemma}~\label{lem:OStests}
  For the sequence of tests $$\phi'_{n,\e} = \In_{ \left\{ \nX:  \left|\frac{n}{\sum_{i=1}^{n}{T_i-T_{i-1}}} - \lambda_0 \right| >  \lambda_0 \sqrt{\frac{n+2}{(n-2)^2}} e^{Cn\e^2} \right\} \cap \left\{ \nX:  \left|\frac{n}{\sum_{i=1}^{n}{E_i-S_i}} - \mu_0 \right| >  \mu_0 \sqrt{\frac{n+2}{(n-2)^2}} e^{Cn\e^2} \right\} },$$ 
  it can be shown that
  \(\bbE_{P^{n}_0}[ \phi'_{n,\e} ] \leq e^{-{K} n\e^2}, \)
  for $C=K/2$.
\end{lemma}

We assume that $\Theta_n(\e)=\Theta=(0,\infty)^2$. Observe that Assumption~\ref{assump:Asf2} is trivially satisfied by the product of Inverse Gamma priors on $\lambda$ and $\mu$.
Next, we show that the prior and the likelihood model satisfy Assumption~\ref{assump:Asf3}. 
\begin{lemma}\label{lem:nv3}
	Fix $ n_2\geq 2$ and any $\rho >1$. Let   		\(A_n := \left\{ \vxi \in \Theta :D_{1+\rho} \left( P_0^n \| P_{\vxi}^n \right) \leq C_3 n \e_n^2 \right\}\), 	where $D_{1+\rho} \left( P_0^n \| P_{\vxi}^n \right)$ is the R\'enyi divergence between $P_0^n$ and  $P_{\vxi}^n$.
    Then for $\e_n^2=\frac{\log n}{n}$ the prior satisfies 
	\vspace{0em}
		$\Pi \{A_n\}  \geq \exp(-n C_2 \e_n^2) , \forall n\geq n_2$,
	with $C_3 > 4 \max\{\alpha_s^{-1}, \alpha_q^{-1}\}$ and $C_2 = 0.5(\alpha_s + \alpha_q) C_3$.
\end{lemma}

The results above verify the  conditions required to establish the convergence  rate  of the optimal staffing level computed using (VBJCCP). However, to explicitly quantify the rate of convergence, we also need to identify a bound on $\eta_n^2$ using Proposition~\ref{prop:eta_n}. Therefore, in the next result, we identify a sequence of distribution in $\mathcal{Q}$ that satisfies Assumption~\ref{assump:Asf11} required for Proposition~\ref{prop:eta_n} to hold.
\begin{lemma}\label{lem:nv6}
	Let $\{Q_n(\lambda,\mu)\}$ be a sequence of distributions defined as $\Gamma(\lambda;n,n/\lambda_0) \Gamma(\mu;n,n/\mu_0)$, then 
	\( \frac{1}{n} \left[ \scKL\left(Q_n(\lambda,\mu)\|\Pi(\theta) \right) + \bbE_{Q_n(\theta)} \left[ \scKL\left(dP^n_{0}(\nX))\| dP^n_{\theta}(\nX) \right) \right]  \right] \leq C_9 \e_n'^2,  \)
where $\e_n'^2= \frac{\log n }{n}$ and $C_9 = 1 + \max \left(0,  2+  \frac{2\beta_q}{\lambda_0}  - \log \sqrt{2\pi }    - \log \left( \frac{{\beta_q}^{\alpha_q}}{\Gamma(\alpha_q)} \right)  + \alpha_q\log \lambda_0 \right) + \max \bigg(0,  2+  \frac{2\beta_s}{\mu_0} $
$- \log \sqrt{2\pi }    - \log \left( \frac{\beta_s^{\alpha_s}}{\Gamma(\alpha_s)} \right)  + \alpha_s\log \mu_0 \bigg)$ and the parameters of the prior distribution are such that $C_9>0$.
	\end{lemma}
	
Lemmas~\ref{lem:nv3} and~\ref{lem:nv6}, combined together, identify that  the optimal staffing level computed using (VBJCCP) converges at the rate of $\e_n= \sqrt{\frac{\log n }{n}}$. More formally,

\begin{theorem}\label{thm:OSfinite}
    For $L^1_n(\vxi,\vxi_0) = n\sup_{c\in \cX} |\In_{(-\infty,0]}(1-W_q(c,\lambda,\mu)-\alpha)-  \In_{(-\infty,0]}(1-W_q(c,\lambda_0,\mu_0)-\alpha) |$ and $L^2_n(\vxi,\vxi_0) = n\sup_{c\in \cX} |c-  c |=0$, where $\cX$ is a finite set of positive integers, 
    there exists a constant $M>0$ (that depends on all the fixed hyper-parameters), such that for any $\eta>0$, 
    \( P_0^n[ |C^*_{VB}(\nX) -  C^*| > 2\eta  ] \leq M \e_n^2, \)
    where $\e_n^2=\frac{\log n}{n}$.
\end{theorem}
	\begin{proof}
	The proof is a direct consequence of Lemmas~\ref{lem:OStests},~\ref{lem:nv3},~\ref{lem:nv6}, Propositions~\ref{prop:eta_n},~\ref{prop:tests}, and Theorem~\ref{thm:finite}. 
	\end{proof}

Using the result above, we can directly establish the following result that quantifies the convergence rate of optimal staffing level computed using (BJCCP) approach.
\begin{corollary}~\label{corr:OSB}
 For $L^1_n(\vxi,\vxi_0) = n\sup_{c\in \cX} |\In_{(-\infty,0]}(1-W_q(c,\lambda,\mu)-\alpha)-  \In_{(-\infty,0]}(1-W_q(c,\lambda_0,\mu_0)-\alpha) |$ and $L^2_n(\vxi,\vxi_0) = n\sup_{c\in \cX} |c-  c |=0$, where $\cX$ is a finite set of positive integers, 
 there exists a constant $\bar M>0$ (that depends on all the fixed hyper parameters), such that for any $\eta>0$, 
    \( P_0^n[ |C^*_{B}(\nX) -  C^*| > 2\eta  ] \leq \bar M \e_n^2, \)
    where $C^*_{B}$ is the optimal staffing level computed using (BJCCP) and  $\e_n^2=\frac{\log n}{n}$.
\end{corollary}
\begin{proof}
    The proof follows straightforwardly from Theorem~\ref{thm:OSfinite} and the fact that $\qnv$ is the same as the true posterior distribution when  the variational family $\mathcal{Q}$ is fixed to all possible distributions.
\end{proof}

Next, we discuss that the prior, the likelihood model, and the variational family easily satisfy Assumptions~\ref{assume:prior},~\ref{assume:lan}, and~\ref{assume:Var}, that are required to show consistency of $C^*_{VB}$. Notice that the prior density $\Pi(\lambda,\mu)=\text{Inv}-\Gamma_{\lambda}(\lambda;\alpha_q,\beta_q)\text{Inv}-\Gamma_{\mu}(\mu;\alpha_s,\beta_s)$ is continuous in $\vxi=\{\lambda,\mu\}$ and places positive mass in the neighbourhood of the true parameter $\vxi_0$  and moreover it is bounded, therefore it satisfies Assumption~\ref{assume:prior}. The exponential models are twice continuously differentiable therefore it satisfies the LAN condition in Assumption~\ref{assume:lan}. Moreover, the variational family, the product of Gamma distributions on $\lambda$ and $\mu$, is absolutely continuous with respect to  the prior distribution and also consists of a sequence of distribution that converges at the true parameter at the rate of $\sqrt{n}$~(refer the construction in~lemma~\ref{lem:nv6}). Therefore, the $\mathcal{Q}$ satisfies Assumption~\ref{assume:Var}. Under  these assumptions, it can be shown using the result in Theorem~\ref{thm:1} that the optimal number of servers computed using (VBJCCP) (and (BJCCP)) are consistent.

\section{Conclusion}
Models of data-driven stochastic optimization have been the subject of a growing body of literature. This paper contributes to this by introducing a Bayesian formulation of a data-driven chance constrained optimization problem. The primary impediment to  practical implementation is the computation of the posterior distribution, which must in almost all circumstances be approximated. This paper advocates for a variational Bayesian (VB) approach to approximate the posterior, both for computational reasons and for ensuring the solution set is convex-feasible (modulo regularity). VB necessarily introduces bias into the estimated posterior expected costs and chance constraints. However, this paper also  rigorously proves asymptotic consistency (in a frequentist sense) and a rate of convergence of the feasible sets and value to those of a `true' constrained optimization problem. \rev{Moreover, in this work, we only addressed the statistical question, which is agnostic to the algorithm used to compute the global VB approximator. Studying the quality of the local VB approximation computed by solving the non-convex ELBO objective together with the subsequent solution of the (VBJCCP) is an interesting future research direction. }

\section{Proofs}~\label{sec:Proof4}

\begin{proof}[Proof of Lemma~\ref{lem:Gcons}]
    First observe that
    \begin{align*}
        \bbE_{\qnv} [L^1_n(\vxi,\vxi_0) ] & = n\bbE_{\qnv} [\sup_{\x\in \cX} |\In_{(-\infty,0]}(g(\x,\vxi))- \In_{(-\infty,0]}(g(\x,\vxi_0)) | ]
        \\
        &\geq n \sup_{\x\in \cX} \bbE_{\qnv}  |\In_{(-\infty,0]}(g(\x,\vxi))- \In_{(-\infty,0]}(g(\x,\vxi_0)) | ]
        \\
        &\geq n \sup_{\x\in \cX} \left| \bbE_{\qnv} [ \left(\In_{(-\infty,0]}(g(\x,\vxi)) - \In_{(-\infty,0]}(g(\x,\vxi_0)) \right)] \right|
        \\
        &= n \sup_{\x\in \cX} \left| Q^* \left(g(\x,\vxi) \leq 0  |\nX \right)  - \In_{(-\infty,0]}(g(\x,\vxi_0)) \right|.
    \end{align*}
    Now using \removet{Theorem}\revt{Lemma}~\ref{thrm:thm1} and the inequality above, it is straightforward to observe that the assertion of the lemma follows.
\end{proof}

\begin{proof}[Proof of Lemma~\ref{lem:Fcons}]
    Proof is similar to Lemma~\ref{lem:Gcons} hence omitted. 
\end{proof}

\begin{proof}[Proof of Proposition~\ref{prop:eta_n}]
The proof follows straightforwardly using the definition of $\eta_n^2$ and Assumption~\ref{assump:Asf11}.
\end{proof}

\begin{proof}[Proof of Proposition~\ref{prop:tests}]
Note that consistent tests always exist for finite-dimensional models on fixed null and alternate sets; for instance, the Kolmogorov-Smirnov test statistic~\cite[Theorem 19.1]{vdV00}. Therefore, the condition of Lemma~\ref{lem:Lecam} is always satisfied for finite dimensional (or parametric) models. Now for distance functions $L_n^1(\vxi,\vxi_0)$ in Theorem~\ref{thm:finite} and  $L_n^i(\vxi,\vxi_0)$ in Theorem~\ref{prop:2} fix $\mathcal{P}_0=P_0^n$ and $\mathcal{P}_1=\{P_{\vxi}^n: L_n^{(\cdot)}(\vxi,\vxi_0)>n\e^2\}$, where we use $L_n^{(\cdot)}$ to reference either  $L_n^1(\vxi,\vxi_0)$ or $L_n^i(\vxi,\vxi_0)$ for brevity. Note that for any $\e\in(0,1]$, $\mathcal{P}_1$ is fixed. Therefore, it follows from Lemma~\ref{lem:Lecam} that for any $\e\in(\e_n,1]$,
\(  \bbE_{P_0^n}[\phi_n] \leq e^{-Kn} \leq e^{-Kn\e^2} \text{ and } \sup_{P^n \in \mathcal{P}_1 } \bbE_{P^n}[1-\phi_n] \leq e^{-Kn} \leq e^{-Kn\e^2}.\)
For $\e>1$, by assumption in the assertion of the proposition we have,
\( \bbE_{P_0^n}[\phi'_{n,\e}] \leq e^{-K n\e^2} \text{ and } \sup_{P^n \in \mathcal{P}_1 } \bbE_{P^n}[1-\phi'_{n,\e}]  = 0 \leq e^{-Kn\e^2},\)
where the second equality follows since $\mathcal{P}_1$ is null set  for $\e>1$.
Therefore, it follows that there exists a test $\phi_{n,\e} = \phi_n \In_{\{\e\in(0,1])\}} + \phi'_{n,\e} \In_{\{\e\in(1,\infty)\}} $ such that distance function $L_n^{(\cdot)}$ satisfies Assumption~\ref{assump:Asf1}.
\end{proof}

\begin{proof}[Proof of Lemma~\ref{lem:pw}]
    Lemma~\ref{lem:VBcons} implies that the VB approximate posterior $\qnv$ is consistent, and it follows from Definition~\ref{def:degen} that for every $\eta>0$, $\int_{\|\vxi-\vxi_0\|>\eta} \qnv d\vxi  \overset{P_0^n}{\to} 0 \text{ as }  n \to \infty. $
    In fact, $\qnv$ converges pointwise to $\delta_{\vxi_0}$ almost everywhere with respect to Lebesgue measure. Consequently, Scheff\'e's lemma~\cite[Corollary 2.30]{vdV00} implies that $\qnv$ converges to $\delta_{\vxi_0}$ in total-variation distance, that is
    \begin{align}
        d_{TV}(\qnv,\delta_{\vxi_0}) = \sup_{A\subseteq \Theta}|Q^*(A|\nX)- \delta_{\vxi_0}(A)|\overset{P_0^n}{\to} 0 \text{ as } n\to \infty,\label{eq:tv}
    \end{align} where for  any set $A\subseteq \Theta$, $Q^*(A|\nX) = \int_{A}\qnv d\vxi $.  Using this observation note that
    \begin{align}
    \nonumber
        \sup_{\x\in \cX}& \left| \int_{\Theta} \prod_{i=1}^{m}\In_{(-\infty,0]}(g_i(\x,\vxi)) \qnv d\vxi - \prod_{i=1}^{m}\In_{(-\infty,0]}(g_i(\x,\vxi_0))  \right| 
        \\
        \nonumber
        &= \sup_{\x\in \cX} \left| Q^*(\cap_{i=1}^{m}\{g_i(\x,\vxi)<0\}  - \delta_{\vxi_0}(\cap_{i=1}^{m}\{g_i(\x,\vxi)<0\}  \right| 
        \\
        \nonumber
        &=\left| Q^*(\cap_{i=1}^{m}\{g_i(\bar\x,\vxi)<0\}  - \delta_{\vxi_0}(\cap_{i=1}^{m}\{g_i(\bar\x,\vxi)<0\}  \right|
        \leq d_{TV}(\qnv,\delta_{\vxi_0}),
    \end{align}
    for some $\bar\x\in \cX$ at which supremum is attained in the RHS of the first equality above. Now the result follows straightforwardly from~\eqref{eq:tv}.
    
\end{proof}

\begin{proof}[Proof of Lemma~\ref{lem:of}]
\noindent \textbf{Part 1: Point-wise convergence}
The proof uses similar ideas as used in the proof of~\cite[Theorem 3.7]{Dupacova1988}.
Fix $x\in\cX$. Due to Assumption~\ref{ass:Cath}(3), $f(\x,\vxi)$ is uniformly integrable with respect to  any $q\in \cQ$, which implies that for $\qnv$ and for any $\epsilon>0$, there exists a compact set $K_{\e}$ such that for all $n\geq 1$
\( \int_{\Theta \backslash K_{\e} } |f(\x,\vxi)|\qnv d\vxi <\e. \) 

Now fix $\gamma_{\e} := \max_{\vxi \in K_{\e} } |f(\x,\vxi)| $. Note that $\gamma_{\e}<+\infty$, since $K_{\e}$ is compact and $f(\x,\cdot)$ is a continuous mapping for any $x\in \cX$. Define $f_{\e}(\x,\vxi)$ be the truncation of $f(\x,\vxi)$, that is  $f_{\e}(\x,\vxi) = \{f(\x,\vxi)~\text{ if } |f(\x,\vxi)|<\gamma_{\e}, 
    \gamma_{\e}~\text{ if } f(\x,\vxi)>\gamma_{\e}, \text{ and }
    -\gamma_{\e}~\text{ if } f(\x,\vxi)<-\gamma_{\e}\}.
    $
It follows from the definition above that $|f_{\e}(\x,\vxi)|\leq  |f(\x,\vxi)| $, which implies that
\(
   \int_{\Theta \backslash K_{\e} } |f_{\e}(\x,\vxi)|\qnv d\vxi <\e .
\)
Note the $f_{\e}(\x,\vxi)$ is bounded and continuous in $\vxi$, therefore, it follows using the definition of weak convergence and Lemma~\ref{lem:VBcons} that
\begin{align}
     \lim_{n\to \infty} \bbE_{\qnv}[f_{\e}(\x,\vxi)] \overset{P_{0}^n}{=}  f_{\e}(\x,\vxi_0).
     \label{eq:pwtrunc}
\end{align}

Next observe that
\begin{align}
\nonumber
    |\bbE_{\qnv}&[f(\x,\vxi)] - f(\x,\vxi_0)| \\
    \nonumber
    &=\left|\bbE_{\qnv}[f(\x,\vxi)] - \bbE_{\qnv}[f_{\e}(\x,\vxi)] + \bbE_{\qnv}[f_{\e}(\x,\vxi)] -f_{\e}(\x,\vxi_0)\right. \\ \nonumber &\qquad\qquad\qquad\qquad\qquad\qquad\qquad\qquad\qquad\qquad\qquad  \left.+f_{\e}(\x,\vxi_0)- f(\x,\vxi_0)\right|
    \\
    \nonumber
    &\leq \left|\bbE_{\qnv}[f(\x,\vxi)] - \bbE_{\qnv}[f_{\e}(\x,\vxi)] \right| + \left| \bbE_{\qnv}[f_{\e}(\x,\vxi)] -f_{\e}(\x,\vxi_0) \right|\\  & \qquad\qquad\qquad\qquad\qquad\qquad\qquad\qquad\qquad\qquad\qquad+\left| f_{\e}(\x,\vxi_0)- f(\x,\vxi_0)\right|.
    ~\label{eq:eqbd}
    \end{align}
    \sloppy
    Now using the definition of $f_{\e}(\x,\vxi)$ note that
    \( \left|\bbE_{\qnv}[f(\x,\vxi)] - \bbE_{\qnv}[f_{\e}(\x,\vxi)] \right|
    = \left|\int_{\Theta\backslash K_{\e}} (f(\x,\vxi)-f_{\e}(\x,\vxi))\qnv d\vxi \right|
    \leq \int_{\Theta\backslash K_{\e}} |f(\x,\vxi)|\qnv d\vxi + \int_{\Theta\backslash K_{\e}} |f_{\e}(\x,\vxi)|\qnv d\vxi 
    \leq 2\e.
    \)
    
    Similarly, $\left| f_{\e}(\x,\vxi_0)- f(\x,\vxi_0)\right| \leq 2\e$, since due to Assumption~\ref{ass:Cath}(3) $\int_{\Theta \backslash K_{\e} } |f(\x,\vxi)|\qnv d\vxi <\e$ is true for all $n\geq 1$ and consequently for $\delta_{\vxi_0}$ as well.
    Hence, substituting the above two observations into~\eqref{eq:eqbd} yields
    \(|\bbE_{\qnv}[f(\x,\vxi)] - f(\x,\vxi_0)|\leq 4\e + \left| \bbE_{\qnv}[f_{\e}(\x,\vxi)] -f_{\e}(\x,\vxi_0) \right|.\)
    Consequently,  it follows for any $\e>0$ that,
        $$P_0^n\left( |\bbE_{\qnv}[f(\x,\vxi)] - f(\x,\vxi_0)| >5\e  \right) \leq P_0^n\left( |\bbE_{\qnv}[f_{\e}(\x,\vxi)] - f_{\e}(\x,\vxi_0)| >\e  \right).$$ 
    Now taking limits $n\to\infty$ on either side of the inequality above, the result follows straightforwardly using the observation in~\eqref{eq:pwtrunc}.

   \noindent \textbf{Part 2: Uniform convergence: }
    
    Since $\cX$ is compact and $f(\x,\vxi_0)$ is continuous in  $\x$, using Corollary 2.2 in~\cite{Newey1991} the uniform convergence follows from point-wise convergence (Part 1) if there exist a bounded sequence $B_n$ 
    and for all $\x_1,\x_2 \in \cX $, $|\bbE_{\qnv}[f(\x_1,\vxi)]-\bbE_{\qnv}[f(\x_2,\vxi)] |\leq B_n \|\x_1-\x_2\|$. Since, $f(\x,\vxi)$ is locally Lipschitz in $\x$  due to Assumption~\ref{ass:Cath}(2), therefore for $\x_1,\x_2 \in \cX $, 
    \begin{align}
    \nonumber
      |\bbE_{\qnv}[f(\x_1,\vxi)]-\bbE_{\qnv}[f(\x_2,\vxi)] | &\leq \bbE_{\qnv}[ |f(\x_1,\vxi) - f(\x_2,\vxi)|] 
      \\
      &\leq \bbE_{\qnv}[K_{\cX}(\vxi)]\|x_1-x_2\|.  
    \end{align}
The uniform convergence follows since by Assumption~\ref{ass:Cath}(2) $\bbE_{\qnv}[K_{\cX}(\vxi)] \leq \bar K_{\cX} $.
    \end{proof}


\begin{proof}[Proof of Lemma~\ref{lem:OStests}]
    Due to independence of arrival and service time distributions, first note that 
   \(
    \bbE_{P^{n}_0}[ \phi'_{n,\e} ] = P^{n}_0 \Big(  \nX:  \left|n^{-1}{\sum_{i=1}^{n}{T_i-T_{i-1}}} - \lambda_0 \right| >  \lambda_0 \sqrt{{(n+2)}{(n-2)^{-2}}} e^{Cn\e^2}   \Big) 
    \times P^{n}_0 \Big(  \nX:  \left|{n^{-1}}{\sum_{i=1}^{n}{E_i-S_i}} - \mu_0 \right| >  \mu_0 \sqrt{{(n+2)}{(n-2)^{-2}}} e^{Cn\e^2}  \Big) .
    \)
Denote $\xi_i=T_i-T_{i-1}$. Using Chebyschev's inequality observe that 
\begin{align}
\nonumber
	 P^{n}_0 \Bigg(  \left|\frac{n}{\sum_{i=1}^{n}{\xi_i}} - \lambda_0 \right| &>  \lambda_0 \sqrt{\frac{n+2}{(n-2)^2}} e^{Cn\e^2}   \Bigg) \leq {\frac{(n-2)^2 e^{-{2C} n\e^2} }{\lambda_0^2(n+2)}} \bbE_{P^{n}_0}  \left|\frac{n}{\sum_{i=1}^{n}{\xi_i}} - \lambda_0 \right|^2 
	\\
	\nonumber
	&= {\frac{(n-2)^2}{\lambda_0^2(n+2)}} e^{-{2C} n\e^2} \bbE_{P^{n}_0} \left[  \left(\frac{n}{\sum_{i=1}^{n}{\xi_i}}\right)^2 + \lambda_0^2  - \left(\frac{2n\lambda_0}{\sum_{i=1}^{n}{\xi_i}}\right) \right].
\end{align}
Now using the fact that the sum of $n$ i.i.d  exponential random variable with rate parameter $\lambda_0$ is Gamma distributed with rate and shape parameter $\lambda_0$ and $n$ (respectively), we obtain that the RHS in the equation above is bounded above by 
\begin{align}
	{\frac{(n-2)^2}{\vxi_0^2(n+2)}} e^{-{2C} n\e^2} \vxi_0^2  \left[ \frac{n^2}{(n-1)(n-2)} +1 - \frac{2n}{n-2} \right]
	\leq  e^{-{2C} n\e^2}.
	\end{align}
Now, choosing $C=K/2$, we have
\(\bbE_{P^{n}_0}[ \phi'_{n,\e} ] \leq e^{-{K} n\e^2}, \)
and the proposition follows.
\end{proof}

\begin{proof}[Proof of Lemma~\ref{lem:nv3}]
Due to independence of arrival and service time distributions, observe that
\begin{align*}
    D_{1+\rho}& \left( P_0^n \| P_{\vxi}^n \right)  = n \frac{1}{\rho} \log \int \left(\frac{dP_{\lambda_0}}{dP_{\lambda}}\right)^\rho dP_{\lambda_0} + n \frac{1}{\rho} \log \int \left(\frac{dP_{\mu_0}}{dP_{\mu}}\right)^\rho dP_{\mu_0}
    \\
    &= n  \left( \log \frac{\lambda_0}{\lambda} + \frac{1}{\rho} \log \frac{\lambda_0}{(\rho + 1)\lambda_0 -\rho \lambda } \right)
     +n  \left( \log \frac{\mu_0}{\mu}  + \frac{1}{\rho} \log \frac{\mu_0}{(\rho + 1)\mu_0 -\rho \mu } \right), 
\end{align*} 
when  $\left((\rho + 1)\lambda_0 -\rho \lambda\right)>0$ and $\left((\rho + 1)\lambda_0 -\rho \lambda\right)>0$, otherwise $D_{1+\rho} \left( P_0^n \| P_{\vxi}^n \right) = \infty$. 
Using the straightforward inequality for two independent random variables $A$ and $B$ that $P(A+B \leq 2c)\geq P(\{A \leq c\} \cup\{ B \leq c \}  )  = P(\{A \leq c\}) P(\{ B \leq c \}  )$, it follows that
  \(  \Pi(D_{1+\rho} \left( P_0^n \| P_{\vxi}^n \right) \leq C_3n\e_n^2) \geq \text{Inv}-\Gamma_{\lambda}(D_{1+\rho} \left( P_{\lambda_0}^n \| P_{\vxi}^n \right) \leq 0.5C_3n\e_n^2)\times
    \text{Inv}-\Gamma_{\mu}(D_{1+\rho} \left( P_{\mu_0}^n \| P_{\vxi}^n \right) \leq 0.5C_3n\e_n^2).\)

Now consider the first term of the product in the RHS of the equation above. Observe that, $D_{1+\rho} \left( P_0^n \| P_{\lambda}^n \right)$ is non-decreasing in $\rho$ (this also follows from non-decreasing property of the R\'enyi divergence  with respect to $\rho$). Therefore, observe that
\begin{align*}
\text{Inv}-\Gamma_{\lambda}(D_{1+\rho} \left( P_0^n \| P_{\lambda}^n \right) \leq 0.5 C_3n\e_n^2)&\geq \text{Inv}-\Gamma_{\lambda}(D_{\infty} \left( P_0^n \| P_{\lambda}^n \right) \leq 0.5 C_3n\e_n^2) 
\\
&= \text{Inv}-\Gamma_{\lambda}\left( \lambda_0e^{-0.5 C_3 \e_n^2}\leq {\lambda} \leq \lambda_0 \right). 
\end{align*}

The  cumulative  distribution function of inverse-gamma distribution  is $\text{Inv}-\Gamma_{\lambda}(\{\lambda< t \} ):= \frac{\Gamma\left(\alpha_q,\frac{\beta_q}{t} \right)}{\Gamma(\alpha_q)}$, where $\alpha_q(>0)$ is the  shape parameter, $\beta_q(>0)$ is the  scale parameter, $\Gamma(\cdot)$ is the Gamma function,  and $\Gamma(\cdot,\cdot)$ is the  incomplete Gamma function. Therefore, it  follows for $\alpha\geq 1$ that
\begin{align} 
    \nonumber
    \text{Inv}-\Gamma_{\lambda}\Big( \lambda_0e^{-0.5 C_3 \e_n^2}\leq {\lambda} &\leq \lambda_0 \Big)
     = \frac{\int_{\beta_q/\lambda_0}^{\beta_q/\lambda_0 e^{0.5 C_3  \e_n^2}} e^{-x} x^{\alpha_q-1}dx }{\Gamma(\alpha_q)}
    \\
    \nonumber
    &\geq \frac{e^{-\beta_q/\lambda_0 e^{0.5 C_3  \e_n^2} + \alpha_q 0.5 C_3\e_n^2 }}{\alpha_q \Gamma(\alpha_q)} \left(  \frac{\beta_q}{\lambda_0}\right)^{\alpha_q} \left[ 1 - e^{-\alpha_q 0.5 C_3  \e_n^2 }  \right] 
    \\
    \nonumber
    & \geq \frac{e^{-\beta_q/\lambda_0 e^{0.5 C_3  }  }}{\alpha_q \Gamma(\alpha_q)} \left(  \frac{\beta_q}{\lambda_0}\right)^{\alpha_q} \left[e^{-\alpha_q{0.5 C_3 n \e_n^2 }} \right] 
\end{align}
where the penultimate inequality follows since $0<\e_n^2<1$ and the last inequality follows from the fact  that, 
$ 1 - e^{-\alpha_q 0.5 C_3  \e_n^2} \geq e^{-\alpha_q{0.5 C_3 n \e_n^2 }}$, for large enough $n $. Also note that, $ 1 - e^{-\alpha_q 0.5 C_3  \e_n^2} \geq e^{-\alpha_q{0.5 C_3 n \e_n^2 }}$ can't hold true for $\e_n^2=1/n$. However, for $\e_n^2=\frac{\log n}{n}$ it holds for any $n\geq 2$ when $\alpha_q C_3 >4$.
Using similar steps as above we can also bound 
    \(\text{Inv}-\Gamma_{\mu}(D_{1+\rho} \left( P_{\mu_0}^n \| P_{\mu}^n \right) \leq 0.5 C_3n\e_n^2)\geq \frac{e^{-\beta_s/\mu_0 e^{0.5 C_3  }  }}{\alpha_s \Gamma(\alpha_s)} \left(  \frac{\beta_s}{\mu_0}\right)^{\alpha_s} \left[e^{-\alpha_s{0.5 C_3 n \e_n^2 }} \right],\)
for $\alpha_s C_3 >4$.
Therefore, substituting the above two results we have for the prior distribution  defined as the product of two inverse-Gamma priors  on $\lambda$ and $\mu$, $C_3> 4\max(\alpha_s^{-1},\alpha_q^{-1})$, $C_2=0.5 (\alpha_q +\alpha_s ) C_3$ and any $\rho>1$ the  result follows for sufficiently large $n$. 
\end{proof}

\begin{proof}[Proof of Lemma~\ref{lem:nv6}]
    Since family $\mathcal{Q}$ contains all product Gamma distributions, observe that $\{q_n(\cdot) \in \mathcal{Q} \} \forall n\geq 1$. First, due to independence of queue and server data observe that
    \begin{align}
    \nonumber
        \scKL&\left(Q_n(\lambda,\mu)\|\Pi(\vxi) \right) + \bbE_{Q_n(\vxi)} \left[ \scKL\left(dP^n_{0}(\nX))\| dP^n_{\vxi}(\nX) \right) \right] 
        \\
        & =  \scKL\left(q_n(\lambda)\|\pi(\lambda) \right) + \bbE_{q_n(\lambda)} \left[ \scKL\left(dP^n_{\lambda_0}(\nX(q)))\| dP^n_{\lambda}(\nX(q)) \right) \right] 
        \label{eq:os2}
        \\
        &\quad +  \scKL\left(q_n(\mu)\|\pi(\mu) \right) + \bbE_{q_n(\mu)} \left[ \scKL\left(dP^n_{\mu_0}(\nX(s)))\| dP^n_{\mu}(\nX(s)) \right) \right] 
        \label{eq:os3},
    \end{align} 
    where  $q_n(\cdot) = \frac{n^{n}}{(\cdot)_0^{n} \Gamma({n})} (\cdot)^{{n}-1}e^{-{n}\frac{(\cdot)}{(\cdot)_0}}$, $\nX(q)$ and $\nX(s)$ denote the data pertaining to arrival and service times respectively, $\pi(\cdot)$ denote the $\text{Inv}-\Gamma_{\cdot}$  prior. Now consider the first term in~\eqref{eq:os2}; using the definition of the $\scKL$ divergence it follows that 
    \begin{align}
        \scKL(q_n(\lambda)\| \pi(\lambda)) = \bbE_{ q_n(\lambda)}[ \log (q_n(\lambda))] - \bbE_{ q_n(\lambda)}[ \log (\pi(\lambda))].
        \label{eq:eqr1}
    \end{align}  
    Substituting $q_n(\lambda)$ in the first term of the equation above and expanding the logarithm term, we obtain
    \begin{align}
        \nonumber
        \bbE_{ q_n(\lambda)}[ \log (q_n(\lambda))] &= ({n}-1) \int \log \lambda \frac{{n}^{n}}{\lambda_0^{n} \Gamma({n})} \lambda^{{n}-1}e^{-{n}\frac{\lambda}{\lambda_0}} d\lambda - {n} + \log \left( \frac{{n}^{n}}{\lambda_0^{n} \Gamma({n})} \right)
        \\
        = - \log \lambda_0 &+ ({n}-1) \int \log \frac{\lambda}{\lambda_0} \frac{{n}^{n}}{\lambda_0^{n} \Gamma({n})} \lambda^{{n}-1}e^{-{n}\frac{\lambda}{\lambda_0}} d\lambda -{n} + \log \left( \frac{{n}^{n}}{ \Gamma({n})} \right).
        \label{eq:eqr2}
    \end{align}
    Now consider the second term in the equation above. Substitute $\lambda = \frac{t \lambda_0}{{n}}$ into the  integral, we have \begin{align}
        \int \log \frac{\lambda}{\lambda_0} \frac{{n}^{n}\lambda^{{n}-1}e^{-{n}\frac{\lambda}{\lambda_0}}}{\lambda_0^{n} \Gamma({n})}  d\lambda  &=   \int \log \frac{t}{{n}}  \frac{1}{ \Gamma({n})} t^{{n}-1}e^{-t} dt
        &\leq \int \left( \frac{t}{{n}}-1  \right) \frac{t^{{n}-1}e^{-t} }{ \Gamma({n})} dt =0.
        \label{eq:eqr2a}
    \end{align} 
    Substituting the above result into~\eqref{eq:eqr2}, we get
    \begin{align}
        \nonumber
       \bbE_{ q_n(\lambda)}[ \log (q_n(\lambda))]
        \leq - \log \lambda_0 - {n} + \log \left( \frac{{n}^{n}}{ \Gamma(n)} \right)
        &\leq - \log \lambda_0 - {n}  + \log \left( \frac{{n}^{n}e^{{n}}}{ \sqrt{2\pi {n}} {n}^{{n}-1}} \right) 
        \\
        &= - \log \sqrt{2\pi } \lambda_0  + \frac{1}{2}\log {n}, 
        \label{eq:eqr3}
    \end{align}
    where the  second inequality uses the fact that $\sqrt{2\pi {n}} {n}^{{n}}e^{-{n}} \leq {n} \Gamma({n}) $. Recall $\pi(\lambda)= \frac{{\beta_q}^{\alpha_q}}{\Gamma({\alpha_q})} \lambda^{-{\alpha_q}-1}e^{-\frac{{\beta_q}}{\lambda}}$. Now consider the second term in~\eqref{eq:eqr1}. Using the definition of inverse-gamma prior and expanding the logarithm function, we have
    \begin{align}
        \nonumber
        -\bbE_{ q_n(\lambda)}&[ \log (\pi(\lambda))] =  - \log \left( \frac{{\beta_q}^{\alpha_q}}{\Gamma({\alpha_q})} \right) + ({\alpha_q}+1) \int \log {\lambda} \frac{{n}^{n}\lambda^{{n}-1}}{\lambda_0^{n} \Gamma({n})} e^{-{n}\frac{\lambda}{\lambda_0}} d\lambda+  \frac{{\beta_q}{n}}{({n}-1)\lambda_0}
        \\
        \nonumber
        =  - \log &\left( \frac{{\beta_q}^{\alpha_q}}{\Gamma({\alpha_q})} \right) + ({\alpha_q}+1) \left[\int \log \frac{\lambda}{\lambda_0} \frac{{n}^{n}\lambda^{{n}-1}}{\lambda_0^{n} \Gamma({n})} e^{-{n}\frac{\lambda}{\lambda_0}} d\lambda + \log \lambda_0 \right]
        +  \frac{{\beta_q}{n}}{({n}-1)\lambda_0} 
        \\
        &\quad \quad \leq - \log \left( \frac{{\beta_q}^{\alpha_q}}{\Gamma({\alpha_q})} \right)  + \frac{{n}\beta_q}{({n}-1)\lambda_0} + ({\alpha_q}+1)\log \lambda_0,
        \label{eq:eqr4}
    \end{align} 
    where the last inequality follows from the  observation  in~\eqref{eq:eqr2a}. Substituting~\eqref{eq:eqr4} and~\eqref{eq:eqr3} into~\eqref{eq:eqr1} and dividing either sides by $n$, we obtain $(\star)=\frac{1}{n}\scKL(q_n(\lambda)\| \pi(\lambda))$
    \begin{align}
        \nonumber
         (\star)  \leq  \frac{1}{n}&\left( - \log \sqrt{2\pi } \lambda_0  + \frac{1}{2}\log {n} - \log \left( \frac{{\beta_q}^{\alpha_q}}{\Gamma({\alpha_q})} \right)  + {\beta_q} \frac{{n}}{({n}-1)\lambda_0} + ({\alpha_q}+1)\log \lambda_0 \right)
        \\
        &= \frac{1}{2}\frac{\log {n}}{n} +  \frac{{\beta_q}}{(n-1)\lambda_0} + \frac{1}{n}\left( - \log \sqrt{2\pi }    - \log \left( \frac{{\beta_q}^{\alpha_q}}{\Gamma({\alpha_q})} \right)  + ({\alpha_q})\log \lambda_0 \right).
        \label{eq:eqr5}
    \end{align}  
    Now, consider the second term in~\eqref{eq:os2}. Since the observations are independent and identically distributed, we obtain 
    \(
        \frac{1}{n} \bbE_{q(\lambda)} \left[ \scKL\left(  dP^{n}_{\lambda_0}\| p(\nX|\lambda) \right) \right]   =  \bbE_{q_n(\lambda)} \left[ \scKL\left(dP_{\lambda_0} \| p(\xi|\lambda) \right) \right].  
    \)
    Now using the expression for \scKL~ divergence between the two exponential distributions, we have
    \begin{align*}
        \frac{1}{n} \bbE_{q_n(\lambda)} \left[ \scKL\left(  dP^{n}_{\lambda_0}\| p(\nX|\lambda) \right) \right]  &= \int \left(\log \frac{\lambda_0}{\lambda}  + \frac{\lambda}{\lambda_0} -1\right) \frac{{n}^{n}\lambda^{{n}-1} }{\lambda_0^{n} \Gamma({n})} e^{-{n}\frac{\lambda}{\lambda_0}} d\lambda 
        \leq \frac{1}{{n}-1},
    \end{align*}
    where second inequality uses the fact that $\log x \leq x-1$.
    The inequality above combined together with~\eqref{eq:eqr5} for $n \geq 2$ implies that
    \begin{align}
        \nonumber
        \frac{1}{n} &\left[ \scKL\left(q_n(\lambda)\|\pi(\lambda) \right) + \bbE_{q_n(\lambda)} \left[ \scKL\left(dP^{n}_{\lambda_0})\| p(\nX|\lambda) \right) \right]  \right] 
        \\
        &\leq \frac{1}{2} \frac{\log n}{n} +  \frac{1}{n}\left( 2+  \frac{2{\beta_q}}{\lambda_0}  - \log \sqrt{2\pi }    - \log \left( \frac{{\beta_q}^{\alpha_q}}{\Gamma({\alpha_q})} \right)  + {\alpha_q}\log \lambda_0 \right) \leq C_9 \frac{\log n }{n}. 
    \end{align}
    where $C'_9:= \frac{1}{2} + \max \left(0,  {2}+  \frac{2{\beta_q}}{\lambda_0}  - \log \sqrt{2\pi }    - \log \left( \frac{{\beta_q}^{\alpha_q}}{\Gamma({\alpha_q})} \right)  + {\alpha_q}\log \lambda_0 \right)$. Now using similar arguments as used for~\eqref{eq:os2}, we can bound~\eqref{eq:os3} as 
    \begin{align}
        \nonumber
        \frac{1}{n} &\left[ \scKL\left(q(\mu)\|\pi(\mu) \right) + \bbE_{q(\mu)} \left[ \scKL\left(dP^{n}_{\mu_0})\| p(\nX|\mu) \right) \right]  \right] 
        \\
        &\leq \frac{1}{2} \frac{\log n}{n} +  \frac{1}{n}\left( 2+  \frac{2{\beta_s}}{\mu_0}  - \log \sqrt{2\pi }    - \log \left( \frac{{\beta_s}^{\alpha_s}}{\Gamma({\alpha_s})} \right)  + {\alpha_s}\log \mu_0 \right) \leq C_9 \frac{\log n }{n}. 
    \end{align}
    where $C''_9:= \frac{1}{2} + \max \left(0,  {2}+  \frac{2{\beta_s}}{\mu_0}  - \log \sqrt{2\pi }    - \log \left( \frac{{\beta_s}^{\alpha_s}}{\Gamma({\alpha_s})} \right)  + {\alpha_s}\log \mu_0 \right)$. 
    Combining the above two results the proposition follows with $\e_n'=\frac{\log n }{n}$, and $C_9= C'_9+C''_9$.
\end{proof}

\bibliographystyle{unsrt}
\bibliography{refs}

\end{document}